\documentclass[12pt,twoside]{amsart}
\usepackage[margin=1in]{geometry}

\usepackage{amsfonts,amsmath,amssymb,amsthm,enumerate,stmaryrd,colonequals,mathrsfs,versions,listings} 
\usepackage[pagebackref=true]{hyperref}
\renewcommand*{\backref}[1]{}
\renewcommand*{\backrefalt}[4]{[{\tiny%
    \ifcase #1 Not cited.%
          \or Cited on page~#2.%
          \else Cited on pages #2.%
    \fi%
    }]}
\usepackage[pdftex]{graphicx}
\usepackage{asymptote}
\usepackage{adjustbox} 

\usepackage[colorinlistoftodos]{todonotes} 

\theoremstyle{plain}

\newtheorem{sidethm}{Side Theorem}[section]

\newtheorem{sideques}[sidethm]{Side Question}

\newtheorem{theorem}{Theorem}[section] 
\newtheorem{lemma}[theorem]{Lemma} 
\newtheorem{proposition}[theorem]{Proposition} 
\newtheorem{corollary}[theorem]{Corollary} 
\newtheorem{question}[theorem]{Question} 

\theoremstyle{definition}

\newtheorem{definition}[theorem]{Definition} 
\newtheorem{example}[theorem]{Example} 

\theoremstyle{remark}

\newtheorem{sidermk}[sidethm]{Side Remark}

\newtheorem{remark}[theorem]{Remark} 

\newtheorem*{obs*}{Observation}

\newtheorem*{prop*}{Proposition}

\newcommand{\FF}{\mathbb{F}}

\newcommand{\QQ}{\mathbb{Q}}

\newcommand{\ZZ}{\mathbb{Z}}

\newcommand{\rank}{\operatorname{rank}}

\newcommand{\Wlog}{without loss of generality}

\newcommand{\ord}{\operatorname{ord}}

\newcommand{\card}[1]{\lvert{#1}\rvert}

\newcommand{\eps}{\epsilon}

\newcommand{\Spec}{\operatorname{Spec}}
\newcommand{\inject}{\operatorname{\hookrightarrow}} 
\newcommand{\surject}{\operatorname{\twoheadrightarrow}} 

\newcommand{\Gal}{\operatorname{Gal}} 
\newcommand{\qchar}[2]{\left(\frac{#1}{#2}\right)} 
\newcommand{\tqchar}[2]{(\frac{#1}{#2})} 
\newcommand{\Cl}{\operatorname{Cl}} 

\DeclareFontFamily{U}{wncy}{}
\DeclareFontShape{U}{wncy}{m}{n}{<->wncyr10}{}
\DeclareSymbolFont{mcy}{U}{wncy}{m}{n}
\DeclareMathSymbol{\Sha}{\mathord}{mcy}{"58}

\newcommand{\map}{\operatorname}
\newcommand{\mcal}{\mathcal}
\newcommand{\mf}{\mathfrak}

\newcommand{\inmatrix}[1]{\left[\begin{smallmatrix} #1 \end{smallmatrix}\right]}

\usepackage{xspace}

\makeatletter
\newcommand{\etale}{\'etal\@ifstar{\'e}{e\xspace}}
\makeatother
\newcommand{\Redei}{R\'{e}dei\xspace}
\newcommand{\Gaschutz}{Gasch\"{u}tz\xspace}
\newcommand{\Frohlich}{Fr\"{o}hlich\xspace}
\newcommand{\Furtwangler}{Furtw\"{a}ngler\xspace}

\begin{document}


\includeversion{arxiv}
\excludeversion{submit}
\includeversion{disclaimer}
\excludeversion{markednewpar}


\title[Martinet's question on $2$-class field towers]{On Hilbert $2$-class fields and $2$-towers of imaginary quadratic number fields}

\date{\today}

\author[Victor Y. Wang]{Victor Y. Wang}
\address{Department of Mathematics, Massachusetts Institute of Technology, \mbox{Cambridge, MA 02139}, United States}
\email{vywang@mit.edu}

\begin{disclaimer}\thanks{v3: incorporated final referee suggestions and corrections; shorter version to appear as J. Number Theory 160 (2016), 492--515. v2: switched to unmarked formatting option in \texttt{versions.sty}, for better readability; now 24 pages. v1: Version for the arXiv, with 25 pages. Shorter version with 19 pages (currently available at \url{https://www.overleaf.com/read/rcfjpsqrxvgt}), with fewer details and SAGE code examples. Comments welcome on either.}\end{disclaimer}
\begin{submit}\thanks{Shorter 
version (currently available at \url{https://www.overleaf.com/read/rcfjpsqrxvgt}), with fewer details and SAGE code examples. Longer version available at \url{http://arxiv.org/abs/1508.06552}.}\end{submit}


\begin{abstract}
Inspired by the Odlyzko root discriminant and Golod--Shafarevich $p$-group bounds, Martinet (1978) asked whether an imaginary quadratic number field $K/\mathbb{Q}$ must always have an infinite Hilbert $2$-class field tower when the class group of $K$ has $2$-rank $4$, or equivalently when the discriminant of $K$ has $5$ prime factors. No negative results are known. Benjamin (2001, 2002) and Sueyoshi (2004, 2009, 2010) systematically established infinite $2$-towers for many $K$ in question, by casework on the associated R\'{e}dei matrices. Others, notably Mouhib (2010), have also made progress, but still many cases remain open, especially when the class group of $K$ has small $4$-rank.

Recently, Benjamin (2015) made partial progress on several of these open matrices when the class group of $K$ has $4$-rank $1$ or $2$. In this paper, we partially address many open cases when the $4$-rank is $0$ or $2$, affirmatively answering some questions of Benjamin. We then investigate barriers to our methods and ask an extension question (of independent interest) in this direction. Finally, we suggest places where speculative refinements of Golod--Shafarevich or group classification methods might overcome the ``near miss'' inadequacies in current methods.
\end{abstract}

\maketitle


\section{Introduction}

We first review some notation and background, and then outline our paper in Section \ref{SUBSEC:own-paper-outline}.

\subsection{Rank, field inclusion, class group, and Hilbert class field notation}

For a prime power $p^i > 1$ and a finitely generated abelian group $A$, let $d_{p^i}(A)\colonequals \dim_{\FF_p} A^{p^{i-1}}/A^{p^i} < \infty$ (which we will often abbreviate as $d_{p^i} A$ for convenience) denote the (generalized) \emph{$p^i$-rank}.

For fields $F$ and $E$, we mean by $F\le E$ (or $E\ge F$) the existence of a field embedding $F\inject E$.

Given a number field $K$, let $\Cl(K)$ and $\Cl^+(K)$ denote the \emph{wide} and \emph{narrow} (or \emph{strict}) ideal class groups, respectively.\begin{arxiv} Specifically, let $I_K$ denote its group of nonzero fractional ideals, $P_K$ the subgroup of principal fractional ideals in $I_K$, and $P^+_K$ the subgroup of \emph{totally positive} fractional ideals in $P_K$. Then recall $\Cl(K) \colonequals I_K/P_K$ and $\Cl^+(K) \colonequals I_K/P^+_K$. The notions coincide when $K$ is totally imaginary, but in general, we only know that $\Cl(K)$ is the quotient $\Cl^+(K) / [P_K / P^+_K]$ of $\Cl^+(K)$. In particular, if $K/\QQ$ is quadratic, then $[P_K:P^+_K]$ equals $1$ if and only if $K$ is imaginary or $x^2 - \Delta_K y^2 = -4$ has an integer solution; otherwise it equals $2$ (a \href{https://en.wikipedia.org/wiki/Fundamental_unit_(number_theory)\#Real_quadratic_fields}{sufficient, but not necessary, criterion} being that $\Delta_K$ has a $3\pmod{4}$ prime divisor).\end{arxiv}

Let $K^1$ (resp. $K_+^1$) be the (resp. \emph{narrow}) \emph{Hilbert class field}\processifversion{arxiv}{ of a number field $K$}, i.e. the maximal abelian extension of $K$ unramified everywhere (resp. outside of infinity). The reciprocity law of ray class field theory gives abelian group isomorphisms $\Gal(K^1 / K) \simeq \Cl(K)$ and $\Gal(K_+^1 / K) \simeq \Cl^+(K)$. Now fix a prime $p$. Let $K_{(p)}^1 \le K^1$ \processifversion{arxiv}{(resp. $K_{+,(p)}^1 \le K_+^1$)} be the \processifversion{arxiv}{(resp. \emph{narrow})} \emph{Hilbert $p$-class field} of $K$, i.e. the maximal abelian \emph{$p$-extension} (i.e. Galois extension with Galois group a $p$-group) of $K$ unramified everywhere\processifversion{arxiv}{ (resp. outside of infinity)}. Then looking at $p$-primary parts yields $\Gal(K_{(p)}^1 / K) \simeq \Cl_p(K)$ \processifversion{arxiv}{and $\Gal(K_{+,(p)}^1 / K) \simeq \Cl^+_p(K)$} by Galois theory.

\subsection{Background: Hilbert class field towers and Golod--Shafarevich}

One may iterate the Hilbert class field construction (i.e. $K^0 \colonequals K$ and $K^{i+1} \colonequals (K^i)^1$ for $i\ge0$) to obtain the \emph{$n$th Hilbert class fields} $K^n$ for $n\ge0$, which together form the \emph{Hilbert class field tower} (which we will refer to as the \emph{$p$-tower}), with \emph{top} $K^\infty \colonequals \bigcup_{n\ge0} K^n$. We call the tower \emph{finite} or \emph{infinite} according as $[K^\infty:K] < \infty$ or $[K^\infty:K] = \infty$; by \cite[Proposition 1]{Roquette}, the tower is finite if and only if $\Cl(L) = 1$ for some (not necessarily unramified or Galois) finite extension $L/K$. We analogously define $K_{(p)}^n$, the \emph{$p$-class field tower}, the \emph{top} $K_{(p)}^\infty$, and \emph{(in)finiteness}; here, the $p$-tower is finite if and only if $\Cl_p(L) = 1$ for some finite extension $L/K$ \cite[Proposition 2]{Roquette}. \begin{arxiv}In particular, these towers are Galois by the maximality of Hilbert ($p$-)class fields.\end{arxiv} (For other properties of class field towers, we refer to \processifversion{arxiv}{Lemmermeyer's survey} \cite{LemmermeyerCFT}, \processifversion{arxiv}{Ershov's survey} \cite{Ershov}, and \processifversion{arxiv}{Koch's book} \cite{Koch}.) For use in Section \ref{SUBSEC:summary-of-previous-methods-progress-Martinet's-question} and Proposition \ref{PROP:Schmithals-2-class-field-splitting-idea}, we state the following standard result, proved by induction.

\begin{proposition}[Cf. Roquette {\cite[proofs of Propositions 1 and 2]{Roquette}}]
\label{PROP:tower-inductive-containment-Roquette}

Fix a number field $K$ and a prime $p$. Let $F/\QQ$ be a finite subfield of $K_{(p)}^\infty/\QQ$. Then $F_{(p)}^\infty \le K_{(p)}^\infty$.

\end{proposition}

\begin{markednewpar}\leavevmode\end{markednewpar}\begin{arxiv}

\begin{proof}
Field $F/\QQ$, being a finite subfield of $K_{(p)}^\infty/\QQ$, must be contained in some finite step $K_{(p)}^n$ of the tower (i.e. $F \le K_{(p)}^n$ for some finite $n\ge0$). Then an inductive argument following \cite[proofs of Propositions 1 and 2]{Roquette} shows that $F_{(p)}^i \le K_{(p)}^{n+i}$ for $i\ge0$, so $F_{(p)}^\infty \le K_{(p)}^\infty$. To illustrate the first step of induction, note that $F_{(p)}^1 / F$ is a $p$-extension (in particular, Galois), so $F_{(p)}^1 K_{(p)}^n$ is a well-defined $p$-extension of $K_{(p)}^n$; this ``relativization'' also preserves unramified-ness (for the nonarchimedean valuations, see e.g. Neukirch's textbook \cite[p. 153, Ch. II, Sec. 7, Proposition 7.2]{Neukirch}).
\end{proof}
\end{arxiv}

In the 1960s, Golod and Shafarevich gave a sufficient criterion---the contrapositive of the following theorem---for the infinitude of $p$-towers.

\begin{theorem}[Vinberg/\Gaschutz refinement of the Golod--Shafarevich inequality \cite{GS}; see \cite{Roquette}]
\label{THM:Golod--Shafarevich}

Fix a number field $K$ and a prime $p$. Suppose $K$ has finite $p$-tower. Then $d_p\Cl(K) < 2 + 2\sqrt{1 + d_p \mcal{O}_K^\times}$, where $\mcal{O}_K^\times$ denotes the unit group of the ring of integers $\mcal{O}_K$.
\end{theorem}

\begin{markednewpar}\leavevmode\end{markednewpar}\begin{arxiv}
\begin{sidermk}
Actually, if $L \colonequals K_{(p)}^\infty$, then $\mcal{O}_K^\times$ can be replaced by the quotient $\mcal{O}_K^\times / N_{L/K}\mcal{O}_L^\times$, which Roquette \cite{Roquette} calls the \emph{norm factor group}. But in practice this seems unwieldy.
\end{sidermk}


\begin{sidermk}[General class field towers]
It would be interesting to try to find an infinite class field tower not ``built'' from an infinite $p$-class field tower for some prime $p$. (To the best of our knowledge, while there are examples of number fields with infinite class field tower but finite $p$-class field tower for all primes $p$, these examples still depend on a \emph{closely related} $q$-class field tower being infinite for some prime $q$.)
\end{sidermk}
\end{arxiv}

It directly follows from Theorem \ref{THM:Golod--Shafarevich} that the $2$-tower is infinite for any imaginary quadratic number field $K$ with $d_2\Cl(K) \ge 5$. Martinet \cite{Martinet} (1978), inspired by the Odlyzko \cite{Odlyzko} (1976) root discriminant bounds, asked whether the same holds when $d_2\Cl(K) = 4$, or equivalently when $K$ has $5$ ramified primes, i.e. the discriminant $\Delta_K$ factors into $5$ prime discriminants $p^*$ (reviewed in Section \ref{SEC:genus-theory} along with $2$-ranks and genus theory).

\subsection{Previous progress on Martinet's question}
\label{SUBSEC:summary-of-previous-methods-progress-Martinet's-question}

No negative results are known. Currently, all of the best (positive) results on Martinet's question, to the author's knowledge, stem from applications of the Golod--Shafarevich bound to unramified $2$-extensions of $K$, combined with genus theory bounds. The point is that $K = \QQ(\sqrt{d})$ has infinite $2$-tower if (and only if) there exists a number field $L/\QQ$ satisfying both of the following conditions.
\begin{enumerate}
\item The well-defined compositum $KL = L(\sqrt{d})$ has infinite $2$-tower. This is guaranteed---in view of Corollary \ref{COR:relative-genus-theory-bounds-2-rank} (below) on relative genus theory---by Theorem \ref{THM:Golod--Shafarevich} (above), if $KL/L = L(\sqrt{d})/L$ is a sufficiently ramified quadratic extension.

\item $KL\le K_{(2)}^\infty$ (in which case $K \le KL\le K_{(2)}^\infty$ yields $K_{(2)}^\infty \le (KL)_{(2)}^\infty \le K_{(2)}^\infty$ by Proposition \ref{PROP:tower-inductive-containment-Roquette} (above), so $(KL)_{(2)}^\infty = K_{(2)}^\infty$). This is equivalent to $L\le K_{(2)}^\infty$.
\end{enumerate}

\begin{remark}
\label{RMK:non-Galois-choices-of-L?}

The second condition is guaranteed, for instance, when $KL/K$ is an everywhere unramified $2$-extension---or equivalently (by solvability of $2$-extensions), when $KL/K$ is Galois \processifversion{arxiv}{(certainly guaranteed if $L/\QQ$ is Galois)} and $L\le K_{(2)}^\infty$. To our knowledge, all of the best results on Martinet's question choose Galois subfields $L/\QQ$ of $K_{(2)}^\infty / K$, so it might be fruitful to look at different choices of $L$.
\end{remark}

These observations immediately lead to the extended criterion Proposition \ref{PROP:main-idea-extension-ramification-quadratic} (below) for infinite $2$-towers. We are not aware of any particularly usable improvements, but we speculate some in Section \ref{SEC:improvements-of-GS?}. 
We now summarize the best previous (positive) results, in terms of the choice of $L$ in Proposition \ref{PROP:main-idea-extension-ramification-quadratic}.

\begin{itemize}
\item Mouhib \cite{Mouhib} (2010), improving on Sueyoshi \cite{Sueyoshi1neg} (2004), gave a uniformly positive answer to Martinet's question when $\Delta_K = p_1^*\cdots p_5^*$ (with $p^*$ as defined in Section \ref{SEC:genus-theory}) has exactly $1$ \emph{negative} prime discriminant. Say $p_5^* < 0$; then Mouhib took $L$ to be---with modifications in some cases---the well-defined decomposition field of $p_5$ in the elementary abelian $2$-extension $\QQ(\sqrt{p_1^*},\ldots,\sqrt{p_4^*})$. However, it seems difficult to extend the technique---which relies on the total realness of $L$---to the cases of $3$ or $5$ negative prime discriminants.

\end{itemize}
Recall that one may also categorize quadratic number fields by the \Redei matrix (see Section \ref{SUBSEC:Redei-matrix-intro})\begin{arxiv}, motivated by the fact that a rational prime $p$ unramified in $\QQ(\sqrt{p_1^*\cdots p_n^*})$ splits if and only if $\tqchar{p_1^*}{p}\cdots\tqchar{p_n^*}{p} = +1$ (even if $p=2$)\end{arxiv}. The following results build on techniques of Martinet \cite{Martinet}, with $L$ a (Galois) subfield (usually quadratic or biquadratic) of the \emph{narrow genus field} $\QQ(\sqrt{p_1^*},\ldots,\sqrt{p_5^*})/\QQ$ of $K/\QQ$.

\begin{itemize}
\item Hajir \cite{Hajir,HajirCorrection} (1996, 2000), improving on a theorem of Koch \cite{KochConcrete} (1969), established infinite $2$-towers when $d_4\Cl(K) \ge 3$, using Ramsey-type analysis of R\'{e}dei matrices with small rank. Lemmermeyer \cite{Lemmermeyer4class} used analogous methods to study \emph{real} quadratic fields.

\item Later, Benjamin \cite{Benjamin01,Benjamin02} (2001, 2002) and Sueyoshi \cite{Sueyoshi1neg,Sueyoshi5neg,Sueyoshi3neg} (2004, 2009, 2010 for one, five, three negative prime discriminants, respectively) systematically established infinite $2$-towers for many $5\times 5$ R\'{e}dei matrix cases, but still left many cases open (see \cite[Table 1]{Ben2015} for details)---especially for small $d_4\Cl(K)$.
\end{itemize}

However, there seems to be more room for exploration. For example, in this paper we focus on Schmithals' idea \cite{Schmithals} (1980) of looking at $L = F_{(2)}^1$ for a quadratic field $F$---see Proposition \ref{PROP:Schmithals-2-class-field-splitting-idea} below---motivated by the decomposition law of class field theory. Recently, by taking $F$ with $4$ prime discriminants (see \cite[Lemma 9]{Ben2015}---where the preceding text ``$K$ is the compositum of $k^1$ and $F$'' should instead be ``$K$ is the compositum of $k$ and $F^1$'', and $F$ should be specified to satisfy ``$d_F\mid d_k$ so that $K/k$ is unramified''---and the more specific applications \cite[Lemmas 10 and 11]{Ben2015}, where to correct a sign error it should say ``$\QQ(\sqrt{-5.11.401})$'' in the second listing of $k$ in \cite[Remark 6]{Ben2015}), Benjamin \cite{Ben2015} (2015) established infinite $2$-towers in certain sub-cases of several of these open matrix cases with $d_4\Cl(K) \in\{1,2\}$, and explained in particular the failure of his methods for $d_4\Cl(K) = 0$.

\subsection{Outline of paper}
\label{SUBSEC:own-paper-outline}

In contrast to Benjamin \cite{Ben2015}, we take $F$ (in Proposition \ref{PROP:Schmithals-2-class-field-splitting-idea}) with $3$ or $2$ prime discriminants to make some progress when $d_4\Cl(K) \in\{0,2\}$, in the hopes of identifying some of the ``most difficult'' remaining cases of Martinet's question. Specifically, we give some new results in Section \ref{SEC:key-lemmas}, and present our concrete applications to open \Redei matrices in Section \ref{SEC:concrete-new-2-towers-results}. This provides affirmative answers to Benjamin's questions \cite[Questions 1, 2, and 5]{Ben2015} on the existence of new (to the best of our knowledge) imaginary quadratic number fields $K$ with $d_2\Cl(K) = 4$ and $d_4\Cl(K) \in\{0,2\}$.



\begin{remark}

In private correspondence, Benjamin informed us that \cite[Question 5]{Ben2015} has minor typos; the correct version is ``Do there exist new imaginary quadratic number fields $k$ with $\rank C_{k,2} = 4$ that have infinite $2$-class field tower in the case when the $4$-rank of $C_k$ is $0$\ldots?'' with a remark that ``such new fields do not satisfy \cite[Lemma 10]{Ben2015}.'' Also, ``Lemma 11'' in \cite[Question 4]{Ben2015} should be replaced with ``Lemma 10''.
\end{remark}

Our attempts at applying the decomposition law (Proposition \ref{PROP:decomposition-law}) to Proposition \ref{PROP:Schmithals-2-class-field-splitting-idea} naturally lead to Section \ref{SEC:prime-splitting-HCF-barrier}, where we investigate barriers to our methods---specifically, ``insufficient'' prime splitting in $2$-class fields---and establish them as consequences of the ``classical principal genus theorem'' over $\QQ$. We also ask an extension question (of independent interest) in this direction.

Section \ref{SEC:genus-theory} reviews, in particular, the usual relative genus theory estimates on $2$-ranks of class groups, with an additional remark (Remark \ref{RMK:improvement-4-rank-relative-genus-theory}) on potentially helpful additional information from the ambiguous class number formula. Section \ref{SEC:improvements-of-GS?} discusses further possible research directions.

\section{Background: prime discriminants, \texorpdfstring{$2$}{2}-class groups, and genus theory}
\label{SEC:genus-theory}

Experts can quickly skim this section for notation and review of prime discriminants, genus theory, and \Redei matrices.

Recall that for any quadratic number field $K$ with discriminant $\Delta_K$ (so that $K = \QQ(\sqrt{\Delta_K})$) and $t$ (finite) ramified primes, we have a unique factorization $\Delta_K = p_1^*\cdots p_t^*$ into $t$ pairwise coprime \emph{prime (power) discriminants} $p_i^*$ (defined so that $2^*\in\{+8,-8,-4\}$ and $p^* = (-1)^{(p-1)/2} p \equiv 1\pmod{4}$ for odd primes $p$). By Gauss' principal genus theorem, the $2$-ranks of the narrow and wide class groups are simply $d_2 \Cl^+(K) = t-1$ and $d_2 \Cl(K) = t - 1 - [\text{$\Delta_K>0$ and $p_i^* < 0$ for some $i$}]$, where $[*]$ is an indicator function (defined to be $1$ if $*$ holds, and $0$ otherwise).

We now review a relative generalization of these results via the ambiguous class number formula, with an additional remark (Remark \ref{RMK:improvement-4-rank-relative-genus-theory}) on potentially helpful additional information.


\subsection{Relative genus theory}

\begin{theorem}[Cf. Lemmermeyer \cite{LemmermeyerAmbig}, Gras {\cite[p. 180, Remark 6.2.3; p. 383, Remark 4.2.4]{Gras}}, Emerton \cite{Emerton}]
\label{THM:ambig-class-number-formula}

If $L/K$ is cyclic, then we have the \emph{ambiguous class number formula}
\[ \card{\Cl(L)^{\Gal(L/K)}} = \card{\Cl(K)} \cdot \frac{\prod_v e_v}{[L:K]\cdot[\mcal{O}_K^\times : \mcal{O}_K^\times\cap N_{L/K}(L^\times)]},\]
where the product of ramification indices $e_v$ runs over all finite and infinite places $v$ of $K$.
\end{theorem}

\begin{markednewpar}\leavevmode\end{markednewpar}\begin{arxiv}
\begin{sidermk}[Resemblance with genus class number formula]

We have the short exact sequence $\Cl(L)^{\Gal(L/K)} \inject \Cl(L) \surject \Cl(L)^{1-\sigma}$ induced by the surjection $I\pmod{P_L} \mapsto I/\sigma(I) \pmod{P_L}$. \emph{It is important here that $L/K$ is cyclic, so that a single generator $\sigma$ generates $\Gal(L/K)$.} In particular, $\card{\Cl(L)}/\card{\Cl(L)^{1-\sigma}} = \card{\Cl(L)^{\Gal(L/K)}}$.

Let $H_K,H_L$ denote Hilbert class fields of $K,L$, and $H{L/K}$ the \emph{relative genus field} of $L$ over $K$ (which, when $L/K$ is cyclic, hence abelian, is just the maximal everywhere unramified extension of $L$ abelian over $K$). Some Galois theory (as in Gras \cite[Remark 4.2.4]{Gras} or Cornell \cite[Propositions 3 and 4]{Cornell}) shows that $\Gal(H_L/K)/\Cl(L)^{1-\sigma}$ is the \emph{abelianization} $\Gal(H_L/K)^{\map{ab}} = \Gal(H_{L/K} / K)$, so that $H_{L/K}$ is the fixed field of (the normal subgroup) $\Cl(L)^{1-\sigma}$. In particular, $[H_{L/K}:K] = [H_L:K] / \card{\Cl(L)^{1-\sigma}} = [L:K][H_L:L]/\card{\Cl(L)^{1-\sigma}} = [L:K] \card{\Cl(L)}/\card{\Cl(L)^{1-\sigma}} = [L:K] \card{\Cl(L)^{\Gal(L/K)}}$. Thus in \emph{relative genus field $H_{L/K}$} language, $\card{\Cl(L)^G} = [H_{L/K}:K]/[L:K] = [H_{L/K}:H_K][H_K:K]/[L:K] = [H_{L/K}:H_K]\cdot \card{\Cl(K)}/[L:K]$---so the \emph{ambiguous class number formula} becomes
\[ [H_{L/K}:H_K] = \frac{\prod_v e_v}{[\mcal{O}_K^\times : \mcal{O}_K^\times\cap N_{L/K}(L^\times)]}. \]
This is close to the \emph{genus class number formula} going back to Gauss, Hilbert, and \Furtwangler, which holds for all abelian extensions $L/K$---see Lemmermeyer's survey \cite[p. 11]{LemmermeyerCFT}.

\end{sidermk}

\begin{sideques}
Is it true that $H_{L/K} / H_K$ is an elementary abelian $2$-extension if $[L:K] = 2$? We would want to show that for any $K$-automorphism $g\colon H_L\to H_L$ fixing $K'$ (i.e. $g|_{K'} = 1$ is trivial), that $g^2$ fixes $H_{L/K}$ (i.e. $g^2|_{H_{L/K}}$ is trivial). We may need the map $\Cl(L)\simeq \Gal(H_L/L) \inject \Gal(H_L/K) \surject \Gal(H_K/K)\simeq \Cl(K)$ to be sufficiently canonical.

\end{sideques}
\end{arxiv}

We do not know the precise origins of the following classical $2$-rank estimates. If $L/K$ is cyclic with Galois group $G$, consider the norm map $\phi\colon \Cl(L) \to \Cl(K)$ (sending $I\pmod{P_L}$ to $N_{L/K}(I)\pmod{P_K}$) and its restriction $\psi\colonequals \phi|_{\Cl(L)^G}$. Then $\ker\psi$ lies inside the $[L:K]$-torsion $\Cl(L)^G[[L:K]]$. Restricting $\psi$ further to the $2$-primary part $\Cl(L)^G[2^\infty]$ induces an injection
\[ \Cl(L)^G[2^\infty] / (\ker\psi)[2^\infty] \inject \Cl_2(K). \]

Now \emph{assume in addition that $L/K$ is quadratic}. Then $\Cl(L)^G[[L:K]] = \Cl(L)^G[2]$ is an elementary abelian $2$-group, so one obtains the following result.

\begin{corollary}[Relative genus theory bounds on $2$-rank, cf. Jehne {\cite[p. 230, Section 5]{Jehne}}, Lemmermeyer {\cite[Proposition 1.3.19]{LemmermeyerCFT}},]
\label{COR:relative-genus-theory-bounds-2-rank}

If $L/K$ is quadratic with Galois group $G$, then
\begin{align}
\label{EQ:ambig-resulting-2-rank-bounds}
d_2 \Cl(L)
\ge d_2 \Cl(L)^G
\ge d_2 \ker\psi
&\ge v_2(\card{\Cl(L)^G}) - v_2(\card{\Cl(K)}) \\
&\underbrace{= \map{ram}(L/K) - 1 - d_2 \mcal{O}_K^\times/ \mcal{O}_K^\times\cap N_{L/K}(L^\times)}_{\text{from Theorem \ref{THM:ambig-class-number-formula} and $(K^\times)^{[L:K] = 2} \le N_{L/K}(L^\times)$}},
\end{align}
where $\map{ram}(L/K)$ denotes the number of finite or infinite primes of $K$ ramified in $L$. Furthermore, $(\ker\phi^{\card{\Cl(K)}/\card{\Cl_2(K)}})[2] \le \Cl(L)^G$.\begin{arxiv} Note that $\Cl(L/K)[2] \le (\ker\phi^{\card{\Cl(K)}/\card{\Cl_2(K)}})[2]$, where $\Cl(L/K)\colonequals \ker\phi$ denotes the \emph{relative class group}.\end{arxiv} In particular, if $h_K$ is odd, i.e. $v_2(\card{\Cl(K)}) = 0$, then $\Cl(L)[2] = \Cl(L)^G[2] = \Cl(L)^G[2^\infty]$ and equality holds everywhere in the inequality \eqref{EQ:ambig-resulting-2-rank-bounds}; when $K = \QQ$ we recover the classical formula for $d_2\Cl(\QQ(\sqrt{\Delta}))$.
\end{corollary}

\begin{remark}[Additional information]
\label{RMK:improvement-4-rank-relative-genus-theory}

Can one get better results (cf. Lemmermeyer \cite[Questions 8 and 9]{LemmermeyerCFT})? For instance, in the equality case $\rank_2\Cl(L) = v_2(\#\Cl(L)^G) - v_2(\#\Cl(K))$, we must have $\ker\psi = \Cl(L)^G[2]$ and $\Cl(L)^G[2^\infty]/\ker\psi \simeq \Cl_2(K)$, so $d_4\Cl(L)^G = d_2\Cl(K)$ in particular, which could be helpful. We speculate further in Section \ref{SEC:improvements-of-GS?}.
\end{remark}

\begin{markednewpar}\leavevmode\end{markednewpar}\begin{arxiv}
\begin{sidermk}
See Lemmermeyer's survey \cite[Proposition 1.3.19]{LemmermeyerCFT}: the kernel $\ker\phi$ is called the relative class group $\Cl(L/K)$. (There seems to be more than one definition in the literature, but this is one of them.) For $K = \QQ$ we have the usual class group, of course. See more facts, especially on the index of the image of $\phi$ in $\Cl(K)$, at \cite[Theorem 1.2.5. Takagi's main theorem for Hilbert class fields]{LemmermeyerCFT} (specifically item 6).
\end{sidermk}

\begin{proof}[Proof of general injection]

An invariant ideal class $I\pmod{P_L}\in \Cl(L)^G$ lies in the kernel $\ker\psi$ if and only if $N_{L/K}(I)\in P_K$. But for $I\pmod{P_L}\in \Cl(L)^G$ we have $N_{L/K}(I)\mcal{O}_L = I^{[L:K]}\mcal{O}_L = I^{[L:K]}$. So \emph{if} $N_{L/K}(I)\in P_K$, then $N_{L/K}(I)\mcal{O}_L\in P_L$, so $I^{[L:K]}\in P_L$. So $\ker\psi$ lies inside the $[L:K]$-torsion subgroup $\Cl(L)^G[[L:K]]$ of $\Cl(L)^G$. (It also lies inside $\ker\phi$, the relative class group $\Cl(L/K)$.)

Consider the exact sequence $1\to \ker\psi \subseteq \Cl(L)^G \to \Cl(K)$. 
Restrict $\psi$ further to the $2$-primary part $\Cl(L)^G[2^\infty]$. Then the restriction $\psi_2$ maps $\Cl(L)^G[2^\infty]$ into the $2^\infty$-torsion $\Cl_2(K)$ of $\Cl(K)$. Clearly $\ker\psi_2$ is just the $2^\infty$-torsion of $\ker\psi$ (because $\ker\psi \subseteq \Cl(L)^G$). Thus $\psi$ induces the desired injection $\Cl(L)^G[2^\infty] / (\ker\psi)[2^\infty] \inject \Cl_2(K)$.
\end{proof}

\begin{proof}[Proof of specialization to {$[L:K] = 2$}, including Inequality \eqref{EQ:ambig-resulting-2-rank-bounds}]

In this case $\Cl(L)^G[[L:K]] = \Cl(L)^G[2]$ is an elementary abelian $2$-group, so its subgroup $\ker\psi$ is also elementary, and we obtain $d_2 \Cl(L) \ge d_2 \Cl(L)^G \ge d_2\ker\psi \ge v_2(\#\Cl(L)^G) - v_2(\#\Cl(K))$, so Theorem \ref{THM:ambig-class-number-formula} and $(K^\times)^{[L:K] = 2} \le N_{L/K}(L^\times)$ (as well as $[L:K] = 2$ and the fact that $e_v\in\{1,2\}$ for all $v$) prove Inequality \eqref{EQ:ambig-resulting-2-rank-bounds}.

For the second part, consider $[I]\in \Cl(L)[2]$, so that $I^2 = a\cdot \mcal{O}_L$ for some $a$. Then $I/\sigma(I) = I^2/I\sigma(I) = a^2 N_{L/K}(I)^{-1} \mcal{O}_L$---somewhat abusing notation (as $N_{L/K}(I)\in \mcal{I}_K$). Raising to exponent of $\card{\Cl(K)}/\card{\Cl_2(K)} \equiv1\pmod{2}$ gives $I/\sigma(I) \in N_{L/K}(I^{\card{\Cl(K)}/\card{\Cl_2(K)}})^{-1} P_L$ (again somewhat abusing notation), since $I^2\in P_L$ (so also $\sigma(I)^2\in P_L$).

In particular, if $[I]\in \Cl(L)[2]$ and $I\in \ker\phi^{\card{\Cl(K)}/\card{\Cl_2(K)}}$, then $I/\sigma(I)\in P_L$, i.e. $[I] = \sigma[I]$, so $[I]\in \Cl(L)^G$. In other words, $(\ker\phi^{\card{\Cl(K)}/\card{\Cl_2(K)}})[2] \le \Cl(L)^G$.
\end{proof}

\begin{proof}[Proof of specialization to {$[L:K] = 2$ and $h_K$ odd}, including equality of Inequality \eqref{EQ:ambig-resulting-2-rank-bounds}]

If, in addition, $h_K$ is odd, i.e. $\card{\Cl_2(K)} = 1$, then $\ker\phi^{\card{\Cl(K)}/\card{\Cl_2(K)}} = \Cl(L)$, so specializing gives $\Cl(L)[2] \le \Cl(L)^G$, so that $\Cl(L)[2] = \Cl(L)^G[2]$ in this case. But trivially $v_2(\#\Cl(K)) = 0$ and $v_2(\#\Cl(L)^G) \ge \rank_2\Cl(L)^G$, so equality indeed holds everywhere in Inequality \eqref{EQ:ambig-resulting-2-rank-bounds}.
\end{proof}
\end{arxiv}

\subsection{\Redei matrices and \texorpdfstring{$4$}{4}-ranks}
\label{SUBSEC:Redei-matrix-intro}

\begin{definition}[Cf. R\'{e}dei \cite{Redei}; R\'{e}dei--Reichardt \cite{RedeiReichardt}]

For a quadratic number field $K$ with prime discriminant factorization $\Delta_K = p_1^*\cdots p_t^*$, let $R_K$ denote the (additive) \emph{\Redei matrix} $[a_{ij}]\in \FF_2^{t\times t}$ (up to re-labeling of the $p_i^*$) with $(-1)^{a_{ij}} \colonequals \tqchar{p_i^*}{p_j}$ when $i\ne j$, and $(-1)^{a_{ii}} \colonequals \tqchar{\Delta_K/p_i^*}{p_i}$ (so that the row vectors sum to $0\in\FF_2^t$). Here $\tqchar{*}{*}$ denotes the Kronecker (not Legendre) symbol, so that $a_{ij} = 0$ if and only if $p_j$ splits in the quadratic field $\QQ(\sqrt{p_i^*})$ (even if $p_j = 2$).
\end{definition}

\begin{remark}
In this paper we will often draw \Redei matrices without the diagonal entries, which can be recovered by the fact that column sums are $0$.
\end{remark}

\begin{markednewpar}\leavevmode\end{markednewpar}\begin{arxiv}
\begin{sidermk}[Relevant version of quadratic reciprocity]
For our purposes, if $p^*,q^*$ are prime discriminants with $p,q$ distinct, then $\tqchar{p^*}{q} = -\tqchar{q^*}{p}$ if $p^*,q^* < 0$ are both negative and not equal to $-4$; and $\tqchar{p^*}{q} = \tqchar{q^*}{p}$ if at least one of $p^*,q^*$ is positive, and neither equals $-4$. If, say, $q^* = -4$, then $p$ is odd and it is easiest just to directly compute $\tqchar{p^*}{2} = [p^*\equiv 1\pmod{8}] = (-1)^{(p^* - 1)/4} = (-1)^{(p^2 - 1)/8}$ and $\tqchar{-4}{p} = [p\equiv1\pmod{4}] = (-1)^{(p-1)/2}$.
\end{sidermk}
\end{arxiv}

\begin{theorem}[$4$-rank of narrow class group; cf. R\'{e}dei \cite{Redei}; R\'{e}dei--Reichardt criterion \cite{RedeiReichardt}]

Let $K/\QQ$ be a quadratic number field with $t$ rational primes dividing $\Delta_K$. Then $d_4\Cl^+(K) = t-1 - \rank_{\FF_2} R_K$.
\end{theorem}

\begin{arxiv}For the connection between $4$-rank and $C_4$-splitting/factorizations of $\Delta_K$, and related matters, see e.g. Lemmermeyer \cite{LemmermeyerCFT,LemmermeyerConstruction,Lemmermeyer4class}, Hurrelbrink \cite{Hurrelbrink}, and Hajir \cite{Hajir,HajirCorrection}. In a different direction, Waterhouse \cite{Waterhouse} gave an $8$-rank criterion (see also Hasse \cite{Hasse} and Lu \cite{Lu}), and Kolster \cite{Kolster} gave a criterion for all powers of $2$. Yue \cite{Yue} generalized the R\'{e}dei--Reichardt $4$-rank criterion to relative quadratic extensions over a base field with odd class number.\end{arxiv}

\section{Key lemmas}
\label{SEC:key-lemmas}

We start by reviewing some background, but refer the reader to Section \ref{SUBSEC:concrete-new-lemmas} for concrete new results, which we will apply to some open sub-cases of Martinet's question in Section \ref{SEC:concrete-new-2-towers-results}.

\subsection{Background}

Let $K/\QQ$ be an imaginary quadratic field extension with $\Delta_K = p_1^*\cdots p_t^* < 0$ (recall $t=5$ in Martinet's question) and $L/\QQ$ a finite subfield of $K_{(2)}^\infty/\QQ$ with $\sqrt{\Delta_K}\notin L/\QQ$, so that $KL=L(\sqrt{\Delta_K})$ is quadratic over $L$. Then
\begin{itemize}
\item $(KL)_{(2)}^\infty = K_{(2)}^\infty$ (from the beginning of Section \ref{SUBSEC:summary-of-previous-methods-progress-Martinet's-question}), so $[K_{(2)}^\infty:K] = \infty$ if and only if $[(KL)_{(2)}^\infty:KL] = \infty$.

\item By Dirichlet's unit theorem, $d_2\mcal{O}_{KL}^\times = (\frac12[KL:\QQ]-1) + 1 = [L:\QQ]$, as $-1\in \mcal{O}_{KL}^\times$ and $KL$ is always totally imaginary.

\item Similarly, $d_2\mcal{O}_L^\times$ equals $\frac12[L:\QQ]$ if $L$ is totally imaginary, and equals $[L:\QQ]$ if $L$ is totally real.

\item By Corollary \ref{COR:relative-genus-theory-bounds-2-rank} applied to the relative quadratic extension $KL/L$, we have
\[ d_2\Cl(KL) \ge \map{ram}(L(\sqrt{\Delta_K})/L) - 1 - d_2(\mcal{O}_L^\times/[\mcal{O}_L^\times\cap N_{KL/L}((KL)^\times)]) \ge \map{ram}(L(\sqrt{\Delta_K})/L) - 1 - d_2 \mcal{O}_L^\times. \]

\end{itemize}

Applying Theorem \ref{THM:Golod--Shafarevich} to $KL$ now yields the following main idea of most relevant papers.

\begin{proposition}[Cf. \cite{Martinet}, \cite{Schmithals}, and \cite{Schoof}]
\label{PROP:main-idea-extension-ramification-quadratic}

With notation as above, the field $KL$, and thus $K$ by extension, has an infinite $2$-tower if any of the following criteria hold:
\begin{itemize}
\item $d_2 \Cl(KL) \ge 2+2\sqrt{1 + [L:\QQ]}$;

\item $\map{ram}(L(\sqrt{\Delta_K})/L) - 1 - d_2(\mcal{O}_L^\times/[\mcal{O}_L^\times\cap N_{KL/L}((KL)^\times)]) \ge 2+2\sqrt{1 + [L:\QQ]}$;

\item $L/\QQ$ is totally imaginary and $\map{ram}(L(\sqrt{\Delta_K})/L) - 1 - \frac12[L:\QQ] \ge 2+2\sqrt{1 + [L:\QQ]}$;

\item $L/\QQ$ is totally real and $\map{ram}(L(\sqrt{\Delta_K})/L) - 1 - [L:\QQ] \ge 2+2\sqrt{1 + [L:\QQ]}$.
\end{itemize}
\end{proposition}

\begin{remark}
Suppose $L/\QQ$ is unramified at $m\ge0$ primes dividing $\Delta_K$, say $p_1,\ldots,p_m$. Then $\map{ram}(L(\sqrt{\Delta_K})/L)$ equals $\#\{\wp\in\Spec\mcal{O}_L: \wp\mid p_1\cdots p_m\}$ if $L$ is totally imaginary, and equals $[L:\QQ] + \#\{\wp\in\Spec\mcal{O}_L: \wp\mid p_1\cdots p_m\}$ if $L$ is totally real (the $[L:\QQ]$ coming from ramification at the infinite places of $L$).
\end{remark}

\begin{markednewpar}\leavevmode\end{markednewpar}\begin{arxiv}
\begin{sidermk}
\label{SIDE:RMK:unit-norm-techniques-Schoof-Lemmermeyer}

The term $d_2(\mcal{O}_L^\times/[\mcal{O}_L^\times\cap N_{KL/L}((KL)^\times)])$---measuring the solvability of $x^2 - \Delta_K y^2 = \eps$ in $L$ for the various units $\eps\in \mcal{O}_L^\times$---seems unwieldy in general (almost all papers use the trivial upper bound $d_2 \mcal{O}_L^\times$), though under certain hypotheses Schoof \cite[Remark 3.5 and Theorem 3.7]{Schoof} and Lemmermeyer \cite[Proof of Theorem 3]{Lemmermeyer4class} have been able to give stronger yet clean estimates this way, on different but related problems.
\end{sidermk}
\end{arxiv}

Schmithals \cite{Schmithals} (1980) introduced the idea (Proposition \ref{PROP:Schmithals-2-class-field-splitting-idea}) of looking at $L = F_{(2)}^1$ for a (quadratic) field $F$. The motivation comes from the decomposition law of class field theory; we will use the following particular $2$-extension version.

\begin{proposition}[Decomposition law and application, consult e.g. {\cite[Theorem 1.2.5]{LemmermeyerCFT}}]
\label{PROP:decomposition-law}

Set $L \colonequals F_{(2)}^1$ for a number field $F$. \begin{arxiv} A (nonzero) prime $\mf{p}$ of $F$ splits into $\card{\Cl(F)}/\ord_{\Cl(F)}[\mf{p}]$ primes in $F^1$, and hence (by Galois theory of $F\le L\le F^1$) into $\card{\Cl_2(F)}/\ord_{\Cl(F)/\Cl(F)[\ZZ\setminus2\ZZ]}[\mf{p}]$ primes in $L$; note that $\ord_{\Cl(F)/\Cl(F)[\ZZ\setminus2\ZZ]}[\mf{p}]$ equals the largest power of $2$ dividing $\ord_{\Cl(F)}[\mf{p}]$. \end{arxiv} For $m\ge0$ distinct rational primes $p_1,\ldots,p_m$, we have
\[ \#\{\wp\in\Spec\mcal{O}_L: \wp\mid p_1\cdots p_m\} = \sum \card{\Cl_2(F)}/(\text{largest power of $2$ dividing $\ord_{\Cl(F)}[\mf{p}]$}), \]
where the sum runs over the primes $\mf{p}$ of $F$ dividing $p_1\cdots p_m$.
\end{proposition}

\begin{remark}[See \cite{Schmithals}]
\label{RMK:inert-principal-decomposition}

If a rational prime $p$ is inert in $F/\QQ$, then (perhaps surprisingly) we still have lots of splitting in $L/\QQ$: the prime ideal $p\mcal{O}_F$ is \emph{principal}, hence \emph{totally split} in $L/F$\processifversion{arxiv}{ (in fact, also in the extension $F^1/F$)}. In fact, from a ``random \Redei matrix'' perspective (when $F$ is quadratic), it is \emph{harder} to guarantee lots of splitting in $L/\QQ$ when $p$ splits in $F/\QQ$, as discussed in Section \ref{SEC:prime-splitting-HCF-barrier}.

\end{remark}

\begin{markednewpar}\leavevmode\end{markednewpar}\begin{arxiv}
\begin{sidermk}
\label{SIDE:RMK:ray-class-field-decomposition-law}

There is an analogous decomposition law for narrow Hilbert class fields (and ray class fields in general), presented for instance in Neukirch's textook \cite[p. 409, Ch. VI, Sec. 6, Theorem 7.3]{Neukirch} and \url{http://math.stackexchange.com/a/1264488/43100}.
\end{sidermk}
\end{arxiv}


\begin{proposition}[$2$-class field idea, cf. Schmithals \cite{Schmithals} (1980) and Schoof \cite{Schoof} (1986)]
\label{PROP:Schmithals-2-class-field-splitting-idea}

Let $K/\QQ$ be an imaginary quadratic field extension with $\Delta_K = p_1^*\cdots p_t^*$ and $F/\QQ$ a finite subfield of $K_{(2)}^\infty/\QQ$, with $F/\QQ$ unramified at $m\ge1$ primes dividing $\Delta_K$, say $p_1,\ldots,p_m$. Then $L \colonequals F_{(2)}^1$ is a finite subfield (by Proposition \ref{PROP:tower-inductive-containment-Roquette}) of $K_{(2)}^\infty/\QQ$ unramified at $p_1,\ldots,p_m$, with $\sqrt{\Delta_K}\notin L/\QQ$ (since $m\ge1$). By Proposition \ref{PROP:main-idea-extension-ramification-quadratic} and $[L:\QQ] = 2[L:F] = 2\card{\Cl_2(F)}$, the field $K$ has an infinite $2$-tower if any of the following criteria hold:
\begin{itemize}
\item $d_2 \Cl(KL) \ge 2+2\sqrt{1 + 2\card{\Cl_2(F)}}$;

\item $\map{ram}(L(\sqrt{\Delta_K})/L) - 1 - d_2(\mcal{O}_L^\times/[\mcal{O}_L^\times\cap N_{KL/L}((KL)^\times)]) \ge 2+2\sqrt{1 + 2\card{\Cl_2(F)}}$;

\item $L/\QQ$ is totally imaginary and $\#\{\wp\in\Spec\mcal{O}_L: \wp\mid p_1\cdots p_m\} - 1 - \card{\Cl_2(F)} \ge 2+2\sqrt{1 + 2\card{\Cl_2(F)}}$;

\item $L/\QQ$ is totally real and $\#\{\wp\in\Spec\mcal{O}_L: \wp\mid p_1\cdots p_m\} - 1 \ge 2+2\sqrt{1 + 2\card{\Cl_2(F)}}$.
\end{itemize}
\end{proposition}

\begin{markednewpar}\leavevmode\end{markednewpar}\begin{arxiv}
\begin{sidermk}
When $F/\QQ$ is Galois, the maximality of $F_{(2)}^1/F = L/F$ proves that $L/\QQ$ is also Galois.
\end{sidermk}

\begin{sidermk}[Relaxing to more general unramified solvable towers]
If instead of strictly looking at $2$-towers, one allows a few ``steps'' to be unramified, then it could turn out to be much easier to get uniform results on infinitude of class field towers. For instance, following Schoof \cite{Schoof}, one could take $L = F^1$ instead of $L = F_{(2)}^1$.
\end{sidermk}

\begin{sidermk}
\label{SIDE:RMK:attempt-Schmithals-for-narrow-class-field?}

In view of Side Remark \ref{SIDE:RMK:ray-class-field-decomposition-law}, it could in principle be useful to consider the variant with $L \colonequals F_{+,(2)}^1$ instead the \emph{narrow} Hilbert $2$-class field. However, we have not yet managed to find a particularly fruitful application to (the open cases of) Martinet's question, because if $F/\QQ$ is quadratic with $\Cl^+(F)\ne \Cl(F)$ (so $F$ must be real in particular), then $F \le F_{(2)}^1 \le F^1$ is a tower of totally real fields---but one can then check that $L$ becomes totally imaginary (in which case the criterion is stricter, due to lack of archimedean ramification).
\end{sidermk}
\end{arxiv}

\subsection{Concrete new lemmas}
\label{SUBSEC:concrete-new-lemmas}

We obtain the following Lemmas \ref{LEM:F=Q(p3*p4*p5*)-imaginary}, \ref{LEM:F=Q(p4*p5*)-real-positive-p4*,p5*}, and \ref{LEM:F=Q(p4*p5*)-imaginary} by applying Proposition \ref{PROP:decomposition-law} to some choice of quadratic field $F/\QQ$ in Proposition \ref{PROP:Schmithals-2-class-field-splitting-idea}. Throughout this section, we denote by $K/\QQ$ an imaginary quadratic field with $\Delta_K = \ell_1^*\cdots \ell_5^*$ (exactly $5$ prime discriminants).

\subsubsection{\texorpdfstring{$F$}{F}: imaginary quadratic field with three prime discriminants}

When $\ell_3^* \ell_4^* \ell_5^* < 0$, taking the imaginary quadratic field $F = \QQ(\sqrt{\ell_3^* \ell_4^* \ell_5^*}) \le K_{(2)}^\infty$ in Proposition \ref{PROP:Schmithals-2-class-field-splitting-idea} yields Lemma \ref{LEM:F=Q(p3*p4*p5*)-imaginary}.

\begin{lemma}
\label{LEM:F=Q(p3*p4*p5*)-imaginary}

If some three of the five prime discriminants, say $\ell_3^*,\ell_4^*,\ell_5^*$, have negative product, then the imaginary quadratic field $K$ has infinite $2$-tower if any of the following criteria hold:
\begin{enumerate}
\item $d_2 \Cl(KL) \ge 2+2\sqrt{1 + 2\card{\Cl_2(F)}}$;
\label{SUBLEM:F=Q(p3*p4*p5*)-imaginary:G-S}

\item $\card{\Cl_2(F)}\ge 16$ and $\ell_1,\ell_2$ are \emph{both} inert in $F/\QQ$\begin{arxiv}, i.e. $\tqchar{\ell_3^*}{p}\tqchar{\ell_4^*}{p}\tqchar{\ell_5^*}{p} = -1$ for \emph{both} of $p=\ell_1,\ell_2$\end{arxiv}.
\label{SUBLEM:F=Q(p3*p4*p5*)-imaginary:16,2-inert}
\end{enumerate}
Here $F = \QQ(\sqrt{\ell_3^* \ell_4^* \ell_5^*})$ and $L \colonequals F_{(2)}^1$ from Proposition \ref{PROP:Schmithals-2-class-field-splitting-idea} are totally imaginary.
\end{lemma}

\begin{remark}
Since $F$ is imaginary, Gauss' genus theory gives $\Cl_2(F) = \Cl^+_2(F) \simeq C_{2^m}\oplus C_{2^n}$ for some $m,n \ge 1$, so the inequality $\card{\Cl_2(F)} \ge 8$ is equivalent to $d_4\Cl_2(F) \ge 1$, or $\rank_{\FF_2} R_F \le 1$ by \Redei--Reichardt. (When $\ell_3^*,\ell_4^*,\ell_5^*$ are all negative and not equal to $-4$, this condition is particularly clean: in this case there are only $2$ types of \Redei matrices $R_F$ up to re-indexing, one with rank $1$ and the other with rank $2$.) Combined with Waterhouse's determination of $d_8\Cl^+_2(F)$ \cite{Waterhouse}, one could in principle obtain a criterion for $\card{\Cl_2(F)} \ge 16$.
\end{remark}

\begin{remark}
In view of the group-theoretic success in (for instance) Koch \cite{KochCentral}, Maire \cite{Maire}, and Benjamin--Lemmermeyer--Snyder \cite{BLS}, it could potentially be enlightening to analyze the ``near miss'' or ``borderline'' cases $\Cl_2(F) \simeq (2,4)$ or $\Cl_2(F) \simeq (2,8)$.
\end{remark}

\begin{example}
For the reader's convenience, we now list our attempts at using Lemma \ref{LEM:F=Q(p3*p4*p5*)-imaginary}, with the first column detailing the number of negative prime discriminants among $\ell_3^*,\ell_4^*,\ell_5^*$ in the application of Lemma \ref{LEM:F=Q(p3*p4*p5*)-imaginary}\eqref{SUBLEM:F=Q(p3*p4*p5*)-imaginary:16,2-inert}, and whether $-4\in\{\ell_3^*,\ell_4^*,\ell_5^*\}$.

\begin{tabular}{c|c|c|c|c|c|c}
Application & Progress? & Examples & $R_K$ & $\#\{p_i^*<0\}$ & $-4\in\{p_i^*\}$? & $d_4\Cl(K)$  \\

\hline
- & none & Ex. \ref{EX:matrix-B-small-examples} & $B$ & $5$ & no & $0$ \\

\hline
$3$ neg., no $-4$ & Thm. \ref{THM:5-neg-matrix-A} & Ex. \ref{EX:matrix-A-examples} & $A$ & $5$ & no & $0$ \\

\hline
$3$ neg., no $-4$ & Thm. \ref{THM:3-neg-matrix-28-alternative-progress} & Ex. \ref{EX:matrix-28-examples-mimicking-matrix-A} & 28 & $3$ & no & $0$ \\

\hline
$3$ neg., no $-4$ & Thm. \ref{THM:5-neg-matrix-C} & Ex. \ref{EX:matrix-C-examples} & $C$ & $5$ & yes & $0$ \\

\hline
$3$ neg., yes $-4$ & Thm. \ref{THM:5-neg-matrix-D1} & Ex. \ref{EX:matrix-D1-examples} & $D_1$ & $5$ & yes & $0$ \\

\hline
$1$ neg., yes $-4$ & Thm. \ref{THM:4-rank-2-family-D2} & Ex. \ref{EX:matrix-(d)-in-D2-examples} & $\mcal{D}_2$ & $3$ & yes & $2$ \\

\end{tabular}

\end{example}


\begin{proof}[Proof of Lemma \ref{LEM:F=Q(p3*p4*p5*)-imaginary}\eqref{SUBLEM:F=Q(p3*p4*p5*)-imaginary:16,2-inert}]

$L/\QQ$ is totally imaginary, so by Proposition \ref{PROP:Schmithals-2-class-field-splitting-idea} it suffices to check $\#\{\wp\in\Spec\mcal{O}_L: \wp\mid \ell_1\ell_2\} - 1 - \card{\Cl_2(F)} \ge 2+2\sqrt{1 + 2\card{\Cl_2(F)}}$.

Since $\ell_1,\ell_2$ are inert in $F/\QQ$, the decomposition law (specifically, the inert trick of Remark \ref{RMK:inert-principal-decomposition}) yields $\#\{\wp\in\Spec\mcal{O}_L: \wp\mid \ell_i\} = [L:F] = \card{\Cl_2(F)}$ for $i=1,2$. Thus
\[ \#\{\wp\in\Spec\mcal{O}_L: \wp\mid \ell_1\ell_2\}
\ge \underbrace{  \card{\Cl_2(F)} + \card{\Cl_2(F)} \ge 3 + \card{\Cl_2(F)} + 2\sqrt{1 + 2\card{\Cl_2(F)}}  }_\text{if $\card{\Cl_2(F)} \ge 7+2\sqrt{11} = 13.6332\ldots$} \]
verifies the desired criterion when $\card{\Cl_2(F)} \ge 16$.
\end{proof}

\subsubsection{\texorpdfstring{$F$}{F}: real quadratic field with exactly two positive prime discriminants}

Alternatively, when $\ell_4^*,\ell_5^* > 0$, taking the real quadratic field $F = \QQ(\sqrt{\ell_4^* \ell_5^*}) \le K_{(2)}^\infty$ in Proposition \ref{PROP:Schmithals-2-class-field-splitting-idea} gives Lemma \ref{LEM:F=Q(p4*p5*)-real-positive-p4*,p5*}.

\begin{remark}
\label{RMK:Schmithals-priority-Mouhib-acknow}

This particular idea seems to originate from Schmithals \cite{Schmithals} (and independently later by Hajir \cite[p. 17, last paragraph]{Hajir} and Mouhib \cite[Proposition 3.3]{Mouhib}, \cite{MouhibAcknowledgement}), who took $F = \QQ(\sqrt{(+5)(+461)})$---with class number $16$---to show that $\QQ(\sqrt{(+5)(-11)(+461)})$, an imaginary quadratic field with $2$-class group $C_4\oplus C_2$, has infinite $2$-tower.
\end{remark}

\begin{lemma}
\label{LEM:F=Q(p4*p5*)-real-positive-p4*,p5*}

If some two of the five prime discriminants, say $\ell_4^*,\ell_5^*$, are positive, then the imaginary quadratic field $K$ has infinite $2$-tower if any of the following criteria hold:
\begin{enumerate}
\item $d_2 \Cl(KL) \ge 2+2\sqrt{1 + 2\card{\Cl_2(F)}}$;
\label{SUBLEM:F=Q(p4*p5*)-real-positive-p4*,p5*:G-S}

\item $\card{\Cl_2(F)}\ge 8$ and \emph{at least $1$} of $\ell_1,\ell_2,\ell_3$ is inert in $F/\QQ$\begin{arxiv}, i.e. $\tqchar{\ell_4^*}{p}\tqchar{\ell_5^*}{p} = -1$ for at least $1$ of $p=\ell_1,\ell_2,\ell_3$\end{arxiv};
\label{SUBLEM:F=Q(p4*p5*)-real-positive-p4*,p5*:8,1-inert}

\item $\card{\Cl_2(F)}\ge 4$ and \emph{at least $2$} of $\ell_1,\ell_2,\ell_3$ is inert in $F/\QQ$\begin{arxiv}, i.e. $\tqchar{\ell_4^*}{p}\tqchar{\ell_5^*}{p} = -1$ for at least $2$ of $p=\ell_1,\ell_2,\ell_3$\end{arxiv}.
\label{SUBLEM:F=Q(p4*p5*)-real-positive-p4*,p5*:4,2-inert}
\end{enumerate}
Here $F = \QQ(\sqrt{p_4^* p_5^*})$ and (by non-ramification of $L/F$ at $\infty$) $L \colonequals F_{(2)}^1$ from Proposition \ref{PROP:Schmithals-2-class-field-splitting-idea} are totally real.
\end{lemma}

\begin{example}
\label{EX:source-of-failure-and-4-rank-direction-motivation}

When Lemma \ref{LEM:F=Q(p4*p5*)-real-positive-p4*,p5*}\eqref{SUBLEM:F=Q(p4*p5*)-real-positive-p4*,p5*:8,1-inert} fails, it is natural to ask (assuming $K$ has infinite $2$-tower) \emph{where} the failure comes from: Golod--Shafarevich, or the genus theory input? For instance, take $K = \QQ(\sqrt{(-7)(-3)(-8)(+29)(+5)})$, which has an open \Redei matrix
\[ R_K = \#49 = \left[\begin{array}{ccc|cc}
- & 1 & 0 & 0 & 1 \\
0 & - & 1 & 1 & 1 \\
1 & 0 & - & 1 & 1 \\
\hline
0 & 1 & 1 & - & 0 \\
1 & 1 & 1 & 0 & -
\end{array}\right]. \]
(For partial positive progress on matrix 49, see Theorem \ref{THM:3-neg-matrix-34a,49}.) Here $F \colonequals \QQ(\sqrt{(+29)(+5)})$ has class number $4$ (as well as narrow class number $4$), so its Hilbert $2$-class field $L \colonequals F_{(2)}^1 = F^1$ coincides with its Hilbert class field, which can be computed in SAGE.

The genus theory input gives a lower bound
\[ d_2\Cl(KL) \ge \#\{\wp\in\Spec\mcal{O}_L: \wp\mid (-7)(-3)(-8)\} - 1 \ge 4+2+2 - 1 = 7\]
---here $7$ is inert in $F/\QQ$, and then splits completely into $4$ primes in $L/F$, while $3,2$ split into $2$ primes in $F/\QQ$ and then stay inert in $L/F$, due to Theorem \ref{THM:CPGT-F=Q(p4*p5*)-real-positive-p4*,p5*}. In fact, here the bound is tight: the class group $\Cl(KL)$ has cyclic direct sum decomposition $(336, 336, 4, 4, 2, 2, 2)$ (under the {\tt proof=False} flag in SAGE, i.e. assuming GRH for a reasonable run-time), so $2$-rank exactly $7$, which is just shy of the $2+2\sqrt{8+1} = 8$ needed for Golod--Shafarevich. But Golod--Shafarevich doesn't take into account the $4$-rank of $4$, or the $8$- and $16$- ranks of $2$, so it would be nice to have a strengthening incorporating such data; see Question \ref{QUES:stronger-4-rank-Golod-Shaf?} for further speculation.


\lstset{language = Python, basicstyle  = \ttfamily, breaklines=true}
\begin{lstlisting}
R.<x> = PolynomialRing(QQ)
F.<d> = (x^2 - 5*29).splitting_field(); F
F.class_group(); F.narrow_class_group()

H.<a> = F.hilbert_class_field(); H
H.class_group(); H.narrow_class_group()

Z.<u> = H.extension(x^2 - 5*29*(-8)*(-3)*(-7)); Y.<v> = Z.absolute_field(); Y
Z.class_group(proof=False); #Z.class_group(proof=True)
\end{lstlisting}

\begin{markednewpar}\leavevmode\end{markednewpar}\begin{arxiv}
Here is the output from SAGE.

\begin{lstlisting}
Number Field in d with defining polynomial x^2 - 145
Class group of order 4 with structure C4 of Number Field in d with defining polynomial x^2 - 145
Multiplicative Abelian group isomorphic to C4
Number Field in a with defining polynomial x^4 - 6*x^2 - 5*x - 1 over its base field
Class group of order 1 of Number Field in a with defining polynomial x^4 - 6*x^2 - 5*x - 1 over its base field
Trivial Abelian group
Number Field in v with defining polynomial x^16 + 193696*x^14 + 20*x^13 + 16471863912*x^12 + 968240*x^11 + 803220850771478*x^10 - 61753673800*x^9 + 24564405504767344470*x^8 - 5925738645060300*x^7 + 482450842595634923493592*x^6 - 178917857880265611700*x^5 + 5942613705132902419583139177*x^4 - 2412654792740884887465160*x^3 + 41972919293354024066614415135974*x^2 - 12387969691342454354951827720*x + 130154117938774456595949675427812841
Class group of order 14450688 with structure C336 x C336 x C4 x C4 x C2 x C2 x C2 of Number Field in u with defining polynomial x^2 + 24360 over its base field
\end{lstlisting}
\end{arxiv}
\end{example}

\begin{remark}
Recall that whether $\card{\Cl^+_2(F)} = \card{\Cl_2(F)}$ is subtle for real quadratic fields $F$. But at least both $\Cl^+_2(F)$ and $\Cl_2(F)$ are cyclic here, so by \Redei--Reichardt, the inequality $\card{\Cl^+_2(F)} \ge 4$ is equivalent to $\tqchar{\ell_4^*}{\ell_5} = \tqchar{\ell_5^*}{\ell_4} = +1$.
\end{remark}

\begin{remark}
In view of the group-theoretic success in (for instance) Koch \cite{KochCentral}, Maire \cite{Maire}, and Benjamin--Lemmermeyer--Snyder \cite{BLS}, it could potentially be enlightening to analyze the ``near miss'' or ``borderline'' cases $\Cl_2(F) \simeq C_2$ and $\Cl_2(F) \simeq C_4$.
\end{remark}

\begin{example}
See the applications and examples under Theorem \ref{THM:3-neg-matrix-34a,49} (progress on matrices 34a and 49, using Lemma \ref{LEM:F=Q(p4*p5*)-real-positive-p4*,p5*}\eqref{SUBLEM:F=Q(p4*p5*)-real-positive-p4*,p5*:8,1-inert}, with Example \ref{EX:matrices-34a,49-examples}), in the case where $\Delta_K \not\equiv4\pmod{8}$ has $3$ negative prime discriminants and $d_4\Cl(K) = 0$.
\end{example}

\begin{proof}[Proof of Lemma \ref{LEM:F=Q(p4*p5*)-real-positive-p4*,p5*}\eqref{SUBLEM:F=Q(p4*p5*)-real-positive-p4*,p5*:8,1-inert}]
$L/\QQ$ is totally real, so by Proposition \ref{PROP:Schmithals-2-class-field-splitting-idea} it suffices to check $\#\{\wp\in\Spec\mcal{O}_L: \wp\mid \ell_1\ell_2\ell_3\} - 1 \ge 2+2\sqrt{1 + 2\card{\Cl_2(F)}}$.

Say $\ell_1$ is inert in $F/\QQ$, so $\#\{\wp\in\Spec\mcal{O}_L: \wp\mid \ell_1\} = [L:F] = \card{\Cl_2(F)}$ by the decomposition law (specifically, the inert trick of Remark \ref{RMK:inert-principal-decomposition}). Furthermore, any prime $p\nmid \ell_4^*\ell_5^*$ splits into $2$ or $4$ primes in $\QQ(\sqrt{\ell_4^*},\sqrt{\ell_5^*})/\QQ$, hence at least that many in the extension $L/\QQ$ (inclusion due to $\ell_4^*,\ell_5^*>0$). Thus
\[ \#\{\wp\in\Spec\mcal{O}_L: \wp\mid \ell_1\ell_2\ell_3\}
\ge \underbrace{  \card{\Cl_2(F)} + 2 + 2 \ge 3 + 2\sqrt{1 + 2\card{\Cl_2(F)}}  }_\text{if $\card{\Cl_2(F)} \ge 3+2\sqrt{3} = 6.4641\ldots$} \]
verifies the desired criterion when $\card{\Cl_2(F)} \ge 8$.
\end{proof}

\begin{proof}[Proof of Lemma \ref{LEM:F=Q(p4*p5*)-real-positive-p4*,p5*}\eqref{SUBLEM:F=Q(p4*p5*)-real-positive-p4*,p5*:4,2-inert}]
\begin{arxiv}$L/\QQ$ is totally real, so by Proposition \ref{PROP:Schmithals-2-class-field-splitting-idea} it suffices to check $\#\{\wp\in\Spec\mcal{O}_L: \wp\mid \ell_1\ell_2\ell_3\} - 1 \ge 2+2\sqrt{1 + 2\card{\Cl_2(F)}}$.

Say $\ell_1,\ell_2$ are inert in $F/\QQ$, so $\#\{\wp\in\Spec\mcal{O}_L: \wp\mid \ell_i\} = [L:F] = \card{\Cl_2(F)}$ for $i=1,2$ by the decomposition law (specifically, the inert trick of Remark \ref{RMK:inert-principal-decomposition}). Furthermore, any prime $p\nmid \ell_4^*\ell_5^*$ splits into $2$ or $4$ primes in $\QQ(\sqrt{\ell_4^*},\sqrt{\ell_5^*})/\QQ$, hence at least that many in the extension $L/\QQ$ (inclusion due to $\ell_4^*,\ell_5^*>0$). \end{arxiv}This time
\[ \#\{\wp\in\Spec\mcal{O}_L: \wp\mid \ell_1\ell_2\ell_3\}
\ge \underbrace{  \card{\Cl_2(F)} + \card{\Cl_2(F)} + 2 \ge 3 + 2\sqrt{1 + 2\card{\Cl_2(F)}}  }_\text{if $\card{\Cl_2(F)} \ge \frac12 (3+2\sqrt{3}) = 3.2320\ldots$} \]
verifies the desired criterion when $\card{\Cl_2(F)} \ge 4$.
\end{proof}

\subsubsection{\texorpdfstring{$F$}{F}: imaginary quadratic field with two prime discriminants}

On the other hand, when $\ell_4^* \ell_5^* < 0$, taking the imaginary quadratic field $F = \QQ(\sqrt{\ell_4^* \ell_5^*}) \le K_{(2)}^\infty$ in Proposition \ref{PROP:Schmithals-2-class-field-splitting-idea} gives Lemma \ref{LEM:F=Q(p4*p5*)-imaginary}.

\begin{lemma}
\label{LEM:F=Q(p4*p5*)-imaginary}
If some two of the five prime discriminants, say $\ell_4^*,\ell_5^*$, have opposite sign, then the imaginary quadratic field $K$ has infinite $2$-tower if any of the following criteria hold:
\begin{enumerate}
\item $d_2 \Cl(KL) \ge 2+2\sqrt{1 + 2\card{\Cl_2(F)}}$;
\label{SUBLEM:F=Q(p4*p5*)-imaginary:G-S}

\item $\card{\Cl_2(F)}\ge 16$ and \emph{at least $2$} of $\ell_1,\ell_2,\ell_3$ is inert in $F/\QQ$\begin{arxiv}, i.e. $\tqchar{\ell_4^*}{p}\tqchar{\ell_5^*}{p} = -1$ for at least $2$ of $p=\ell_1,\ell_2,\ell_3$\end{arxiv};
\label{SUBLEM:F=Q(p4*p5*)-imaginary:16,2-inert}

\item $\card{\Cl_2(F)}\ge 4$, \emph{at least $1$} of $\ell_1,\ell_2,\ell_3$ is inert in $F/\QQ$, and \emph{at least $1$} of $\ell_1,\ell_2,\ell_3$ splits completely in $L/\QQ$.
\label{SUBLEM:F=Q(p4*p5*)-imaginary:4,1-inert,1-split}
\end{enumerate}
Here $F = \QQ(\sqrt{p_4^* p_5^*})$ and $L \colonequals F_{(2)}^1$ from Proposition \ref{PROP:Schmithals-2-class-field-splitting-idea} are totally imaginary.
\end{lemma}

\begin{example}
See the applications and examples under Theorems \ref{THM:3-neg-matrix-32} (progress on matrix 32, using Lemma \ref{LEM:F=Q(p4*p5*)-imaginary}\eqref{SUBLEM:F=Q(p4*p5*)-imaginary:16,2-inert}, with Example \ref{EX:matrix-32-examples}) and \ref{THM:3-neg-matrix-16,28} (progress on matrices 16 and 28, using Lemma \ref{LEM:F=Q(p4*p5*)-imaginary}\eqref{SUBLEM:F=Q(p4*p5*)-imaginary:4,1-inert,1-split}, with Example \ref{EX:matrices-16,28-examples}), both in the case where $\Delta_K \not\equiv4\pmod{8}$ has $3$ negative prime discriminants and $d_4\Cl(K) = 0$.
\end{example}

\begin{remark}
Since $F$ is imaginary, $\Cl_2(F) = \Cl^+_2(F) \simeq C_{2^n}$ for some $n\ge1$, so by \Redei--Reichardt, the inequality $\card{\Cl_2(F)} \ge 4$ is equivalent to $\tqchar{\ell_4^*}{\ell_5} = \tqchar{\ell_5^*}{\ell_4} = +1$.
\end{remark}

\begin{proof}[Proof of Lemma \ref{LEM:F=Q(p4*p5*)-imaginary}\eqref{SUBLEM:F=Q(p4*p5*)-imaginary:16,2-inert}]
$L/\QQ$ is totally imaginary, so by Proposition \ref{PROP:Schmithals-2-class-field-splitting-idea} it suffices to check $\#\{\wp\in\Spec\mcal{O}_L: \wp\mid \ell_1\ell_2\ell_3\} - 1 - \card{\Cl_2(F)} \ge 2+2\sqrt{1 + 2\card{\Cl_2(F)}}$.

Say $\ell_1,\ell_2$ are inert in $F/\QQ$, so $\#\{\wp\in\Spec\mcal{O}_L: \wp\mid \ell_i\} = [L:F] = \card{\Cl_2(F)}$ for $i=1,2$ by the decomposition law (specifically, the inert trick of Remark \ref{RMK:inert-principal-decomposition}). Furthermore, any prime $p\nmid \ell_4^*\ell_5^*$ splits into $2$ or $4$ primes in $\QQ(\sqrt{\ell_4^*},\sqrt{\ell_5^*})/\QQ$, hence at least that many in the extension $L/\QQ$ (inclusion due to $F$ imaginary). Thus
\[ \#\{\wp\in\Spec\mcal{O}_L: \wp\mid \ell_1\ell_2\ell_3\}
\ge \underbrace{  \card{\Cl_2(F)} + \card{\Cl_2(F)} + 2 \ge 3 + \card{\Cl_2(F)} + 2\sqrt{1 + 2\card{\Cl_2(F)}}  }_\text{if $\card{\Cl_2(F)} \ge 5+2\sqrt{7} = 10.2915\ldots$} \]
verifies the desired criterion when $\card{\Cl_2(F)} \ge 16$.
\end{proof}

\begin{proof}[Proof of Lemma \ref{LEM:F=Q(p4*p5*)-imaginary}\eqref{SUBLEM:F=Q(p4*p5*)-imaginary:4,1-inert,1-split}]
$L/\QQ$ is totally imaginary, so by Proposition \ref{PROP:Schmithals-2-class-field-splitting-idea} it suffices to check $\#\{\wp\in\Spec\mcal{O}_L: \wp\mid \ell_1\ell_2\ell_3\} - 1 - \card{\Cl_2(F)} \ge 2+2\sqrt{1 + 2\card{\Cl_2(F)}}$.

Say $\ell_1$ is inert in $F/\QQ$ and $\ell_2$ splits completely in $L/\QQ$, so $\#\{\wp\in\Spec\mcal{O}_L: \wp\mid \ell_1\} = [L:F] = \card{\Cl_2(F)}$ by the decomposition law, and $\#\{\wp\in\Spec\mcal{O}_L: \wp\mid \ell_2\} = [L:\QQ] = 2\card{\Cl_2(F)}$. Furthermore, any prime $p\nmid \ell_4^*\ell_5^*$ splits into $2$ or $4$ primes in $\QQ(\sqrt{\ell_4^*},\sqrt{\ell_5^*})/\QQ$, hence at least that many in the extension $L/\QQ$ (inclusion due to $F$ imaginary). Thus
\[ \#\{\wp\in\Spec\mcal{O}_L: \wp\mid \ell_1\ell_2\ell_3\}
\ge \underbrace{  \card{\Cl_2(F)} + 2\card{\Cl_2(F)} + 2 \ge 3 + \card{\Cl_2(F)} + 2\sqrt{1 + 2\card{\Cl_2(F)}}  }_\text{if $\card{\Cl_2(F)} \ge \frac12 (3+2\sqrt{3}) = 3.2320\ldots$} \]
verifies the desired criterion when $\card{\Cl_2(F)} \ge 4$.
\end{proof}

\section{Application to Martinet's question}
\label{SEC:concrete-new-2-towers-results}



We now apply the lemmas from Section \ref{SUBSEC:concrete-new-lemmas} to several sub-cases of open \Redei matrices $R_K$ of rank $2$ or $4$ (corresponding by \Redei--Reichardt to $d_4\Cl(K) = 2$ or $d_4 \Cl(K) = 0$, respectively). We will use the labeling of \Redei matrices from Sueyoshi \cite{Sueyoshi5neg,Sueyoshi3neg} and Benjamin \cite{Ben2015}. In this section, we often write \Redei matrices without the diagonal entries, which can be recovered by the fact that column sums are $0$.

\subsection{\texorpdfstring{$4$-rank $2$}{4-rank 2}}
\label{SUBSEC:4-rank-2}

When $K$ (with five prime discriminants $p_i^*$) has $d_4\Cl(K) = 2$, there is exactly one family of open \Redei matrices, referred to as ``Family $\mcal{D}_2$'' (we use a different font to avoid confusion with matrix $D_2$ in Section \ref{SUBSEC:5-neg-4-rank-0-is-4-mod-8}) by Benjamin \cite[pp. 127--128]{Ben2015} (and falling under ``Case 60'' in Sueyoshi \cite[p. 181, with discussion on p. 184]{Sueyoshi3neg}; note that Benjamin has a minor typo in his listing of the Kronecker symbols in the first paragraph of \cite[p. 127, Section 4. Case 1]{Ben2015}); this is originally due to Benjamin \cite{Benjamin02} (2002). More precisely, the family $\mcal{D}_2$ consists of $K$ with exactly three negative discriminants $p_1^*,p_2^*,p_3^*$ (up to re-indexing) and $\Delta_K \equiv 4\pmod{8}$---so say $p_1^* = -4$, up to re-indexing---and \Redei matrix $R_K$ of the form
\[
 \left[\begin{array}{c|cc|cc}
- & 1 & 1 & 0 & 0 \\
\hline
a_{21} & - & 1 & 1 & 1 \\
a_{31} & 0 & - & 1 & 1 \\
\hline
1 & 1 & 1 & - & a_{45} \\
1 & 1 & 1 & a_{54} (= a_{54}) & -
\end{array}\right],
\]
with four specific possibilities (following Benjamin \cite[p. 128]{Ben2015}):
\begin{enumerate}[(a)]
\item $(a_{45},a_{54}) = (0,0)$ and $(a_{21},a_{31}) = (1,0)$ (so in particular, $a_{11} = 1$);

\item $(a_{45},a_{54}) = (0,0)$ and $(a_{21},a_{31}) = (0,1)$ (so in particular, $a_{11} = 1$);

\item $(a_{45},a_{54}) = (1,1)$ and $(a_{21},a_{31}) = (1,0)$ (so in particular, $a_{11} = 1$);

\item $(a_{45},a_{54}) = (1,1)$ and $(a_{21},a_{31}) = (1,1)$ (so in particular, $a_{11} = 0$).
\end{enumerate}

Benjamin makes (positive) progress on matrices (a) \cite[p. 128, Example 3]{Ben2015}, (b) \cite[p. 128, Example 4]{Ben2015}, and (c) \cite[p. 128, Examples 1 and 2]{Ben2015}---but not (d). We now present some further progress on the $\mcal{D}_2$ family.

\begin{theorem}[Further progress on family $\mcal{D}_2$]
\label{THM:4-rank-2-family-D2}

Suppose $R_K \in \mcal{D}_2$, and set $F \colonequals \QQ(\sqrt{p_1^* p_4^* p_5^*})$ (imaginary quadratic field, so $\Cl(F) = \Cl^+(F)$). Then $p_2,p_3$ are inert in $F/\QQ$.

\Redei--Reichardt says $d_4 \Cl^+_2(F) = (3-1) - 1 = 1$ (since $R_F = \left[\begin{smallmatrix}- & 0 & 0 \\ 1 & - & \alpha \\ 1 & \alpha & -\end{smallmatrix}\right] = \left[\begin{smallmatrix}0 & 0 & 0 \\ 1 & \alpha & \alpha \\ 1 & \alpha & \alpha\end{smallmatrix}\right]$, where $\alpha \colonequals a_{45} = a_{54}$), so by Gauss' genus theory, $\Cl_2(F) \simeq C_2\oplus C_{2^n}$ for some $n\ge2$. By Lemma \ref{LEM:F=Q(p3*p4*p5*)-imaginary}\eqref{SUBLEM:F=Q(p3*p4*p5*)-imaginary:16,2-inert}, $K$ has infinite $2$-tower if $\card{\Cl_2(F)} \ge 16$, i.e. $n\ge3$, holds.
\end{theorem}

\begin{remark}
In the remaining case $\Cl_2(F) \simeq C_2\oplus C_4$ (for family $\mcal{D}_2$), it is plausible that group-theoretic methods could give helpful additional structure for the $2$-tower of $K$.
\end{remark}




\begin{example}[Examples with $R_K = (d)\in \mcal{D}_2$]
\label{EX:matrix-(d)-in-D2-examples}

Start with any three prime discriminants $p_1^* = -4$ and $p_4^*,p_5^* > 0$ (necessarily $p_4,p_5\equiv1\pmod{4}$) such that $R_F = \left[\begin{smallmatrix}- & 0 & 0 \\ 1 & - & 1 \\ 1 & 1 & -\end{smallmatrix}\right]$. It is then standard (Chinese remainder theorem and Dirichlet's theorem) to find primes $p_2,p_3\equiv3\pmod4$ such that $p_2^*,p_3^*\equiv1\pmod{4}$ satisfy the correct Kronecker symbols with respect to each other and with respect to $p_1,p_4,p_5$.

For instance, based on \url{http://oeis.org/A046013} (listing imaginary $F$ with $\#\Cl(F) = 16$), one could take $F = \QQ(\sqrt{-740})$, i.e. $(p_1^*,p_4^*,p_5^*) = (-4,+5,+37)$\begin{arxiv} (as $\tqchar{+5}{2} = \tqchar{+37}{2} = -1$ and $\tqchar{+37}{5} = -1$)\end{arxiv}.
\end{example}

\subsection{Five negative prime discriminants, \texorpdfstring{$4$-rank $0$}{4-rank 0}, discriminant not \texorpdfstring{$4\pmod{8}$}{4 mod 8}}
\label{SUBSEC:5-neg-4-rank-0-not-4-mod-8}

When $K$ has five negative prime discriminants $p_i^*$ all not equal to $-4$, and $d_4\Cl(K) = 0$, there are two open \Redei matrices, called $A$ (resp. (k)) and $B$ (resp. (l)) in \cite[p. 137, Section 8. Case 5]{Ben2015} (resp. \cite[p. 335]{Sueyoshi5neg}):
\[
A = \begin{bmatrix}- & 1 & 1 & 0 & 1 \\ 0 & - & 1 & 1 & 0 \\ 0 & 0 & - & 1 & 1 \\ 1 & 0 & 0 & - & 1 \\ 0 & 1 & 0 & 0 & -\end{bmatrix};
\quad
B = \begin{bmatrix}- & 1 & 1 & 0 & 0 \\ 0 & - & 1 & 1 & 0 \\ 0 & 0 & - & 1 & 1 \\ 1 & 0 & 0 & - & 1 \\ 1 & 1 & 0 & 0 & -\end{bmatrix}.
\]

\begin{remark}

We have made progress on $A$ (Theorem \ref{THM:5-neg-matrix-A}) but not $B$, which ``looks harder'' to us. Note that $B$ is a circulant matrix, so it is certainly harder to exploit any asymmetries.
\end{remark}

\begin{example}[Relatively small examples of $B$]
\label{EX:matrix-B-small-examples}
$R_K = B$, for instance, when the prime discriminant tuple $(p_i^*)_{i\in\ZZ/5\ZZ}$ equals
\begin{itemize}
\item $(-3,-8,-23,-7,-19)$;
\item $(-3,-11,-8,-7,-31)$; or
\item $(-3,-11,-q,-7,-31)$, where $q$ denotes a prime $107\pmod{3\cdot11\cdot7\cdot31\cdot4}$, e.g. $q = 107 + n\cdot 3\cdot11\cdot7\cdot31\cdot4$ for $n=0,1,3,9,10,21,23,33,34,\ldots$.
\end{itemize}
Also, if $F \colonequals \sqrt{(-31)(-3)(-11)} = \sqrt{-1023}$ or $F \colonequals \sqrt{(-7)(-19)(-3)} = \sqrt{-399}$ accordingly (in the spirit of Proposition \ref{PROP:Schmithals-2-class-field-splitting-idea}), then $\card{\Cl(F)} = 16$ (so $\Cl(F) \simeq C_2\oplus C_8$) according to \url{http://oeis.org/A046013}, which may be helpful for computations.
\end{example}

\begin{theorem}[Progress on matrix $A$]
\label{THM:5-neg-matrix-A}

Suppose $R_K = A$, and set $F \colonequals \QQ(\sqrt{p_3^* p_4^* p_5^*})$ (imaginary quadratic field, so $\Cl(F) = \Cl^+(F)$). Then $a_{3j}+a_{4j}+a_{5j} = 1$ (in $\FF_2$) for $j=1,2$, i.e. $p_1,p_2$ are inert in $F/\QQ$.

\Redei--Reichardt says $d_4 \Cl^+_2(F) = (3-1) - 1 = 1$ (since $R_F = \left[\begin{smallmatrix}- & 1 & 1 \\ 0 & - & 1 \\ 0 & 0 & -\end{smallmatrix}\right] = \left[\begin{smallmatrix}0 & 1 & 1 \\ 0 & 1 & 1 \\ 0 & 0 & 0\end{smallmatrix}\right]$), so by Gauss' genus theory, $\Cl_2(F) \simeq C_2\oplus C_{2^n}$ for some $n\ge2$. By Lemma \ref{LEM:F=Q(p3*p4*p5*)-imaginary}\eqref{SUBLEM:F=Q(p3*p4*p5*)-imaginary:16,2-inert}, $K$ has infinite $2$-tower if $\card{\Cl_2(F)} \ge 16$, i.e. $n\ge3$, holds.

\end{theorem}

\begin{remark}
In the remaining case $\Cl_2(F) \simeq C_2\oplus C_4$ (for matrix $A$), it is plausible that group-theoretic methods could give helpful additional structure for the $2$-tower of $K$.
\end{remark}

\begin{example}[Examples with $R_K = A$]
\label{EX:matrix-A-examples}

Start with any three negative prime discriminants $p_3^*,p_4^*,p_5^*\ne -4$ such that $\card{\Cl_2(F)} \ge 8$, or equivalently $\rank_{\FF_2} R_F \le 1$, i.e. $R_F = \left[\begin{smallmatrix}- & 1 & 1 \\ 0 & - & 1 \\ 0 & 0 & -\end{smallmatrix}\right]$ up to re-indexing. It is then standard to find primes $p_1,p_2\equiv3\pmod4$ such that $p_1^*,p_2^*\equiv1\pmod{4}$ satisfy the correct Kronecker symbols with respect to each other and with respect to $p_3,p_4,p_5$.

For instance, based on \url{http://oeis.org/A046013} (listing imaginary $F$ with $\#\Cl(F) = 16$), one could take $F = \QQ(\sqrt{-399})$ (i.e. $(p_3^*,p_4^*,p_5^*) = (-7,-19,-3)$\begin{arxiv}, as $\tqchar{-19}{7} = +1$, $\tqchar{-3}{7} = +1$, and $\tqchar{-19}{3} =-1$\end{arxiv}) or $F = \QQ(\sqrt{-1023})$ (i.e. $(p_3^*,p_4^*,p_5^*) = (-31,-3,-11)$).
\end{example}

\subsection{Five negative prime discriminants, \texorpdfstring{$4$-rank $0$}{4-rank 0}, discriminant \texorpdfstring{$4\pmod{8}$}{4 mod 8}}
\label{SUBSEC:5-neg-4-rank-0-is-4-mod-8}

When $K$ has five negative prime discriminants $p_i^*$ with $p_1^* = -4$ (up to re-indexing), and $d_4\Cl(K) = 0$, there are three open \Redei matrices, called $C$ (resp. (o)) and $D$ (resp. (p)) in \cite[p. 138, Section 8. Case 5]{Ben2015} (resp. \cite[p. 336]{Sueyoshi5neg}), with (following Benjamin) $D$ split into two cases:
\[
C = \begin{bmatrix}- & 1 & 1 & 1 & 1 \\ a_{21} & - & 1 & 0 & 1 \\ a_{31} & 0 & - & 1 & 1 \\ a_{41} & 1 & 0 & - & 1 \\ 1 & 0 & 0 & 0 & -\end{bmatrix};
\quad
D_1 = \begin{bmatrix}- & 1 & 1 & 1 & 1 \\ a_{21} & - & 1 & 1 & 0 \\ 1 & 0 & - & 1 & 1 \\ 0 & 0 & 0 & - & 1 \\ 1 & 1 & 0 & 0 & -\end{bmatrix};
\quad
D_2 = \begin{bmatrix}- & 1 & 1 & 1 & 1 \\ a_{21} & - & 1 & 1 & 0 \\ 0 & 0 & - & 1 & 1 \\ 1 & 0 & 0 & - & 1 \\ a_{51} & 1 & 0 & 0 & -\end{bmatrix}.
\]

\begin{remark}
Benjamin has ambiguous typos in his $D_1,D_2$, with $a_{31} = a_{41}$ (but see ``(p) [$d_4\Cl(K) = 1$] if and only if $a_{31} = a_{41}$'' from Sueyoshi), so our correction \emph{arbitrarily} associates $D_1$ with $(a_{31},a_{41}) = (1,0)$ and $D_2$ with $(a_{31},a_{41}) = (0,1)$. We thank the anonymous referee for observing that Sueyoshi \cite[p. 337]{Sueyoshi5neg} has resolved the case $(a_{31},a_{41},a_{51}) = (1,0,0)$ in matrix $D$ (or (p) in Sueyoshi's terminology).
\end{remark}

\begin{remark}
We have made progress on $C$ (Theorem \ref{THM:5-neg-matrix-C}) and $D_1$ (Theorem \ref{THM:5-neg-matrix-D1}) but not $D_2$, which ``looks hardest'' to us. In some sense $C$ and $D_1$ closely resemble $A$ from Section \ref{SUBSEC:5-neg-4-rank-0-not-4-mod-8}, while $D_2$ closely resembles $B$.
\end{remark}

\begin{theorem}[Progress on matrix $C$]
\label{THM:5-neg-matrix-C}

Suppose $R_K = C$ (in particular, $p_1^* = -4$).

\begin{enumerate}

\item If $a_{31} + a_{41} = 0$, and we set $F \colonequals \QQ(\sqrt{p_3^* p_4^* p_5^*})$, then $p_1,p_2$ are inert in $F/\QQ$, and $R_F = \left[\begin{smallmatrix}- & 1 & 1 \\ 0 & - & 1 \\ 0 & 0 & -\end{smallmatrix}\right] = \left[\begin{smallmatrix}0 & 1 & 1 \\ 0 & 1 & 1 \\ 0 & 0 & 0\end{smallmatrix}\right]$ has rank $1$. 

\item If $a_{21} + a_{41} = 0$, and we set $F \colonequals \QQ(\sqrt{p_2^* p_4^* p_5^*})$, then $p_1,p_3$ are inert in $F/\QQ$, and $R_F = \left[\begin{smallmatrix}- & 0 & 1 \\ 1 & - & 1 \\ 0 & 0 & -\end{smallmatrix}\right] = \left[\begin{smallmatrix}1 & 0 & 1 \\ 1 & 0 & 1 \\ 0 & 0 & 0\end{smallmatrix}\right]$ has rank $1$.

\item If $a_{21} + a_{31} = 0$, and we set $F \colonequals \QQ(\sqrt{p_2^* p_3^* p_5^*})$, then $p_1,p_4$ are inert in $F/\QQ$, and $R_F = \left[\begin{smallmatrix}- & 1 & 1 \\ 0 & - & 1 \\ 0 & 0 & -\end{smallmatrix}\right] = \left[\begin{smallmatrix}0 & 1 & 1 \\ 0 & 1 & 1 \\ 0 & 0 & 0\end{smallmatrix}\right]$ has rank $1$.
\end{enumerate}
These three cases are not mutually exclusive, but for any $K$ with $R_K = C$, at least one will apply, since it is impossible to have $a_{u1}+a_{v1} = 1$ in $\FF_2$ for all pairs $u,v\in\{2,3,4\}$.

In each case, $F$ is imaginary, so $\Cl(F) = \Cl^+(F)$ has $4$-rank $(3-1) - 1 = 1$ by \Redei--Reichardt, and thus $\Cl_2(F) \simeq C_2\oplus C_{2^n}$ for some $n\ge2$. By Lemma \ref{LEM:F=Q(p3*p4*p5*)-imaginary}\eqref{SUBLEM:F=Q(p3*p4*p5*)-imaginary:16,2-inert}, $K$ has infinite $2$-tower if $\card{\Cl_2(F)} \ge 16$, i.e. $n\ge3$, holds in any of the three cases applying to $K$.
\end{theorem}

\begin{example}[Examples with $R_K = C$]
\label{EX:matrix-C-examples}

Start with any three negative prime discriminants $p_3^*,p_4^*,p_5^*\equiv-3\pmod{8}$ such that $\card{\Cl_2(F)} \ge 8$, or equivalently $\rank_{\FF_2} R_F \le 1$, i.e. $R_F = \left[\begin{smallmatrix}- & 1 & 1 \\ 0 & - & 1 \\ 0 & 0 & -\end{smallmatrix}\right]$ up to re-indexing; then $p_1^*=-4,p_3^*,p_4^*,p_5^*$ satisfy permitted \Redei sub-matrix (specifically, with $a_{31} = a_{41} = 1$, which lies in the first case of Theorem \ref{THM:5-neg-matrix-C}). It is then standard to find a prime $p_2\equiv3\pmod4$ such that $p_2^*\equiv1\pmod{4}$ satisfies the correct Kronecker symbols with respect to $p_1=2,p_3,p_4,p_5$.

For instance, based on \url{http://mathworld.wolfram.com/ClassNumber.html} (listing imaginary $F$ with $\#\Cl(F) = 16$), one could take $F = \QQ(\sqrt{-2211})$ (i.e. $(p_3^*,p_4^*,p_5^*) = (-67,-3,-11)$\begin{arxiv}, as $\tqchar{-67}{3} = -1$, $\tqchar{-67}{11} = -1$, and $\tqchar{-11}{3} =+1$\end{arxiv}). We thank Ian Whitehead for providing this example.
\end{example}

\begin{theorem}[Progress on matrix $D_1$]
\label{THM:5-neg-matrix-D1}

Suppose $R_K = D_1$ (in particular, $p_1^* = -4$), and set $F \colonequals \QQ(\sqrt{p_1^* p_3^* p_4^*})$ (imaginary quadratic field, so $\Cl(F) = \Cl^+(F)$). Then $p_2,p_5$ are inert in $F/\QQ$.

\Redei--Reichardt says $d_4 \Cl^+_2(F) = (3-1) - 1 = 1$ (since $R_F = \left[\begin{smallmatrix}- & 1 & 1 \\ 1 & - & 1 \\ 0 & 0 & -\end{smallmatrix}\right] = \left[\begin{smallmatrix}1 & 1 & 1 \\ 1 & 1 & 1 \\ 0 & 0 & 0\end{smallmatrix}\right]$), so $\Cl_2(F) \simeq C_2\oplus C_{2^n}$ for some $n\ge2$. By Lemma \ref{LEM:F=Q(p3*p4*p5*)-imaginary}\eqref{SUBLEM:F=Q(p3*p4*p5*)-imaginary:16,2-inert}, $K$ has infinite $2$-tower if $\card{\Cl_2(F)} \ge 16$, i.e. $n\ge3$, holds.
\end{theorem}

\begin{example}[An infinite family with $R_K = D_1$]
\label{EX:matrix-D1-examples}

Fix $n\ge2$. Lopez \cite[Proposition 1.1]{Lopez} proves that any imaginary quadratic field $E = \QQ(\sqrt{(-4)(-q_3)(-q_4)})$ with $q_3\equiv11\pmod{24}$, $q_4\equiv 7\pmod{24}$, and $q_3+q_4 = 2(3m^2)^{2^{n-1}}$ for some odd integer $m$ has $2$-class group exactly $\Cl_2(F) \simeq C_2\oplus C_{2^n}$, and also that there are infinitely many such quadratic fields $E$ \cite[Theorem 1.3]{Lopez}. For such $E$ we easily check $R_E = \left[\begin{smallmatrix}- & 1 & 1 \\ 1 & - & 1 \\ 0 & 0 & -\end{smallmatrix}\right]$, so we may find $K$ with $(p_1^*,p_3^*,p_4^*) = (-4,-q_3,-q_4)$ (here $F = E$).

\begin{arxiv}Alternatively, based on \url{http://oeis.org/A046013} (listing imaginary $F$ with $\#\Cl(F) = 16$), one could take $F = \QQ(\sqrt{-1876})$ (i.e. $(p_1^*,p_3^*,p_4^*) = (-4,-67,-7)$, as $\tqchar{-7}{2} = +1$, $\tqchar{-67}{2} = -1$, and $\tqchar{-67}{7} = \tqchar{+3}{7} = -1$) or $F = \QQ(\sqrt{-2004})$ (i.e. $(p_1^*,p_3^*,p_4^*) = (-4,-3,-167)$).\end{arxiv}

It is then standard to find primes $p_2,p_5\equiv 3\pmod{4}$ such that $p_2^*,p_5^*\equiv1\pmod{4}$ satisfy the correct Kronecker symbols with respect to each other and $p_1,p_3,p_4$.
\end{example}

\subsection{Three negative prime discriminants, \texorpdfstring{$4$-rank $0$}{4-rank 0}, discriminant not \texorpdfstring{$4\pmod{8}$}{4 mod 8}}
\label{SUBSEC:3-neg-4-rank-0-not-4-mod-8}

When $K$ has five prime discriminants $p_i^*$ all not equal to $-4$, exactly three negative prime discriminants (say $p_1^*,p_2^*,p_3^* < 0 < p_4^*,p_5^*$), and $d_4\Cl(K) = 0$, there are seven open \Redei matrices, numbered 16, 28, 30, 32, 34, 49 (from \cite[p. 140, Section 9. Case 6]{Ben2015} or \cite[pp. 179--180]{Sueyoshi3neg}), with (following Benjamin) 34 split into two cases:
\begin{align*}
\#16 = \left[\begin{array}{ccc|cc}
- & 1 & 1 & 1 & 0 \\
0 & - & 1 & 0 & 1 \\
0 & 0 & - & 1 & 0 \\
\hline
1 & 0 & 1 & - & 1 \\
0 & 1 & 0 & 1 & -
\end{array}\right];
\quad
\#28 = \left[\begin{array}{ccc|cc}
- & 1 & 1 & 0 & 1 \\
0 & - & 1 & 1 & 1 \\
0 & 0 & - & 0 & 1 \\
\hline
0 & 1 & 0 & - & 1 \\
1 & 1 & 1 & 1 & -
\end{array}\right];
\quad
\#32 = \left[\begin{array}{ccc|cc}
- & 1 & 1 & 1 & 1 \\
0 & - & 1 & 0 & 1 \\
0 & 0 & - & 1 & 0 \\
\hline
1 & 0 & 1 & - & 1 \\
1 & 1 & 0 & 1 & -
\end{array}\right];
\displaybreak[0] \\
\#34a = \left[\begin{array}{ccc|cc}
- & 1 & 1 & 1 & 1 \\
0 & - & 1 & 0 & 1 \\
0 & 0 & - & 1 & 1 \\
\hline
1 & 0 & 1 & - & 0 \\
1 & 1 & 1 & 0 & -
\end{array}\right];
\quad
\#49 = \left[\begin{array}{ccc|cc}
- & 1 & 0 & 0 & 1 \\
0 & - & 1 & 1 & 1 \\
1 & 0 & - & 1 & 1 \\
\hline
0 & 1 & 1 & - & 0 \\
1 & 1 & 1 & 0 & -
\end{array}\right];
\quad
\#30 = \left[\begin{array}{ccc|cc}
- & 1 & 1 & 0 & 1 \\
0 & - & 1 & 1 & 0 \\
0 & 0 & - & 1 & 1 \\
\hline
0 & 1 & 1 & - & 1 \\
1 & 0 & 1 & 1 & -
\end{array}\right];
\displaybreak[0] \\
\#34b = \left[\begin{array}{ccc|cc}
- & 1 & 1 & 1 & 1 \\
0 & - & 1 & 0 & 1 \\
0 & 0 & - & 1 & 1 \\
\hline
1 & 0 & 1 & - & 1 \\
1 & 1 & 1 & 1 & -
\end{array}\right].
\end{align*}

\begin{remark}
Benjamin has a single typo in matrix 16 ($a_{15}$ should be $0$, not $1$) and several typos in matrix 32.
\end{remark}

\begin{remark}
We have made progress on matrices 16 and 28 (Theorems \ref{THM:3-neg-matrix-16,28} and \ref{THM:3-neg-matrix-28-alternative-progress}), matrix 32 (Theorem \ref{THM:3-neg-matrix-32}), and matrices 34a and 49 (Theorem \ref{THM:3-neg-matrix-34a,49}), but not matrices 30 and 34b.
\end{remark}




\begin{theorem}[Progress on matrices 34a and 49]
\label{THM:3-neg-matrix-34a,49}

Suppose $R_K\in \{\#34a,\#49\}$, and set $F \colonequals \QQ(\sqrt{p_4^* p_5^*})$ (real quadratic field). Then $p_2$ is inert in $F/\QQ$ if $R_K = \#34a$, and $p_1$ is inert in $F/\QQ$ if $R_K = \#49$.

\Redei--Reichardt says $d_4 \Cl^+_2(F) = (2-1) - 0 = 1$ (since $R_F = \left[\begin{smallmatrix}- & 0 \\ 0 & -\end{smallmatrix}\right] = \left[\begin{smallmatrix}0 & 0 \\ 0 & 0\end{smallmatrix}\right]$), so $\Cl^+_2(F) \simeq C_{2^n}$ for some $n\ge2$ (so $\Cl_2(F)$ is cyclic of order either $2^n$ or $2^{n-1}$). By Lemma \ref{LEM:F=Q(p4*p5*)-real-positive-p4*,p5*}\eqref{SUBLEM:F=Q(p4*p5*)-real-positive-p4*,p5*:8,1-inert}, $K$ has infinite $2$-tower if $\card{\Cl_2(F)} \ge 8$ holds.
\end{theorem}

\begin{example}[Examples for 34a and 49]
\label{EX:matrices-34a,49-examples}

Start with any two positive prime discriminants $p_4^*,p_5^*$ such that $\rank_{\FF_2} R_F = 0$, i.e. $R_F = \left[\begin{smallmatrix}- & 0 \\ 0 & -\end{smallmatrix}\right]$. It is then standard to find primes $p_1,p_2,p_3\equiv3\pmod4$ such that (the three negative prime discriminants) $p_1^*,p_2^*,p_3^*\equiv1\pmod{4}$ satisfy the correct Kronecker symbols with respect to each other and with respect to $p_4,p_5$.

For instance, based on \url{http://oeis.org/A081364} or \url{http://oeis.org/A218158} (to find suitable real $F$ with $\#\Cl(F) = 8$), one could take $F = \QQ(\sqrt{906})$ (i.e. $(p_4^*,p_5^*) = (+8,+113)$) or $F = \QQ(\sqrt{2605})$ (i.e. $(p_4^*,p_5^*) = (+5,+521)$).
\end{example}

\begin{example}[Cf. Remark \ref{RMK:Schmithals-priority-Mouhib-acknow}]

Schmithals \cite{Schmithals} (and independently, Hajir \cite[p. 17, last paragraph]{Hajir} and Mouhib \cite[Proposition 3.3]{Mouhib}, \cite{MouhibAcknowledgement}) used $F = \QQ(\sqrt{(+5)(+461)})$ (note $\qchar{+5}{461} = \qchar{+461}{5} = +1$) with $\#\Cl(F) = 16$ (in particular $\card{\Cl_2(F)} \ge 16$, a stronger assumption than $\card{\Cl_2(F)} \ge8$) to prove that $\QQ(\sqrt{(-11)(+5)(+461)}$ has infinite $2$-tower. As a corollary, cases 34a and 49 already had examples with proven infinite $2$-towers at the time of Benjamin's paper \cite{Ben2015}, but still 34a and 49 remain open in general.
\end{example}

\begin{theorem}[Progress on matrix 32]
\label{THM:3-neg-matrix-32}

Suppose $R_K = \#32$, and set $F \colonequals \QQ(\sqrt{p_3^* p_5^*})$ (imaginary quadratic field, so $\Cl^+(F) = \Cl(F)$). Then $p_1,p_2$ are inert in $F/\QQ$.

\Redei--Reichardt says $d_4 \Cl^+_2(F) = (2-1) - 0 = 1$ (since $R_F = \left[\begin{smallmatrix}- & 0 \\ 0 & -\end{smallmatrix}\right] = \left[\begin{smallmatrix}0 & 0 \\ 0 & 0\end{smallmatrix}\right]$), so $\Cl_2(F) \simeq C_{2^n}$ for some $n\ge2$. By Lemma \ref{LEM:F=Q(p4*p5*)-imaginary}\eqref{SUBLEM:F=Q(p4*p5*)-imaginary:16,2-inert}, $K$ has infinite $2$-tower if $\card{\Cl_2(F)} \ge 16$, i.e. $n\ge4$, holds.
\end{theorem}

\begin{example}[An infinite family with $R_K = \#32$]
\label{EX:matrix-32-examples}

Fix $n\ge1$. Dominguez--Miller--Wong \cite{DMW} prove that any imaginary quadratic field $E = \QQ(\sqrt{(-q_3)(+q_5)})$ with $q_3\equiv3\pmod{8}$, $q_5\equiv 5\pmod{8}$, and $q_5+q_3 = 4(2M^2)^{2^{n-1}}$ for some odd integer $M$ has $2$-class group exactly $\Cl_2(F) \simeq C_{2^n}$, and also that there are infinitely many such quadratic fields $E$ \cite[Theorem 3.1]{DMW}. For such $E$ we easily check that $R_E = \left[\begin{smallmatrix}- & 0 \\ 0 & -\end{smallmatrix}\right]$ if and only if $n\ge2$, so we may find $K$ with $(p_3^*,p_5^*) = (-q_3,+q_5)$ (here $F = E$).

It is then standard to find primes $p_1,p_2,p_4$ with $p_1,p_2\equiv3\pmod{4}$ and $p_4\equiv1\pmod4$ such that $p_1^*,p_2^* < 0$ and $p_4^* > 0$ satisfy the correct Kronecker symbols with respect to each other and with respect to $p_3,p_5$.
\end{example}

\begin{theorem}[Progress on matrices 16 and 28]
\label{THM:3-neg-matrix-16,28}

Suppose $R_K \in\{\#16,\#28\}$.

\begin{enumerate}

\item If $R_K = \#16$, then set $F \colonequals \QQ(\sqrt{p_3^* p_5^*})$.

\item If $R_K = \#28$, then set $F \colonequals \QQ(\sqrt{p_3^* p_4^*})$.
\end{enumerate}
In each case, $p_2$ is inert in $F/\QQ$, the \Redei matrix $R_F = \left[\begin{smallmatrix}- & 0 \\ 0 & -\end{smallmatrix}\right] = \left[\begin{smallmatrix}0 & 0 \\ 0 & 0\end{smallmatrix}\right]$ has rank $0$, and $F$ is imaginary, so $\Cl(F) = \Cl^+(F)$ has $4$-rank $(2-1) - 0 = 1$ by \Redei--Reichardt, and thus $\Cl_2(F) \simeq C_{2^n}$ for some $n\ge2$. By Lemma \ref{LEM:F=Q(p4*p5*)-imaginary}\eqref{SUBLEM:F=Q(p4*p5*)-imaginary:4,1-inert,1-split}, $K$ has infinite $2$-tower if $p_1$ splits completely in $L \colonequals F_{(2)}^1$ (note that $\card{\Cl_2(F)} \ge 4$ automatically holds, from $n\ge2$).
\end{theorem}

\begin{example}
\label{EX:matrices-16,28-examples}

We describe the method for $R_K = \#16$; the case $R_K = \#28$ is analogous. For $R_K = \#16$, one starts with any appropriate $p_3^*,p_5^*\ne -4$ (so $p_3^* < 0 < p_5^*$) satisfying $\tqchar{p_3^*}{p_5} = \tqchar{p_5^*}{p_3} = +1$ (so $\card{\Cl_2(F)} \ge 4$ automatically), and then takes $p_1\equiv3\pmod{4}$ totally split in $L$, guaranteeing $\tqchar{p_3^*}{p_1} = \tqchar{p_5^*}{p_1} = +1$ in view of the tower $F \le \QQ(\sqrt{p_3^*},\sqrt{p_5^*}) \le L$.

Such $p_1$ exists by Chebotarev's density theorem: $\sqrt{-1}\notin F$ means $\sqrt{-1}\notin F_{\map{gen}} = \QQ(\sqrt{p_3^*},\sqrt{p_5^*})$, so $F(\sqrt{-1})/F$ is ramified, so $\sqrt{-1}\notin L$. So $L/\QQ < L(i)/\QQ$ is a \emph{proper} inclusion of Galois extensions, whence by Neukirch \cite[p. 548, Ch. VII, Sec. 13, Corollary 13.10]{Neukirch} there exist infinitely many primes $p_1$ splitting completely in $L$ but \emph{not} splitting completely in $L(i)$.\begin{arxiv} (Of course, any prime splitting completely in $L(i)$ must automatically split completely in $L$.) There is also a more direct approach: $\Gal(L(i)/\QQ) = \Gal(L/\QQ)\times \Gal(\QQ(i)/\QQ)$, so we may apply Chebotarev to the conjugacy class of size $1$ consisting of the generator of $\Gal(\QQ(i)/\QQ)$.\end{arxiv}

It is then standard to find primes $p_2,p_4$ with $p_2\equiv3\pmod{4}$ and $p_4\equiv1\pmod4$ such that $p_2^* < 0$ and $p_4^* > 0$ satisfy the correct Kronecker symbols with respect to each other and with respect to $p_1,p_3,p_5$.
\end{example}

\begin{theorem}[Alternative progress on matrix 28, analogous to Theorem \ref{THM:5-neg-matrix-A} for matrix $A$]
\label{THM:3-neg-matrix-28-alternative-progress}

Suppose $R_K = \#28$, and set $F \colonequals \QQ(\sqrt{p_1^* p_2^* p_3^*})$ (imaginary quadratic field, so $\Cl(F) = \Cl^+(F)$). Then $p_4,p_5$ are inert in $F/\QQ$.

\Redei--Reichardt says $d_4 \Cl^+_2(F) = (3-1) - 1 = 1$ (since $R_F = \left[\begin{smallmatrix}- & 1 & 1 \\ 0 & - & 1 \\ 0 & 0 & -\end{smallmatrix}\right] = \left[\begin{smallmatrix}0 & 1 & 1 \\ 0 & 1 & 1 \\ 0 & 0 & 0\end{smallmatrix}\right]$), so $\Cl_2(F) \simeq C_2\oplus C_{2^n}$ for some $n\ge2$. By Lemma \ref{LEM:F=Q(p3*p4*p5*)-imaginary}\eqref{SUBLEM:F=Q(p3*p4*p5*)-imaginary:16,2-inert}, $K$ has infinite $2$-tower if $\card{\Cl_2(F)} \ge 16$, i.e. $n\ge3$, holds.

\end{theorem}

\begin{example}
\label{EX:matrix-28-examples-mimicking-matrix-A}

The examples here are analogous to those (Example \ref{EX:matrix-A-examples}) for matrix $A$ from the case of five negative discriminants.
\end{example}

\section{Prime splitting barrier, and extension question of independent interest}
\label{SEC:prime-splitting-HCF-barrier}

This section (including Question \ref{QUES:CPGT-based-conjecture-prime-splitting-Hilbert-class-fields} of independent interest) is motivated by our failed attempts at applying (to the open cases of Martinet's question) Proposition \ref{PROP:decomposition-law} and Remark \ref{RMK:inert-principal-decomposition} on the decomposition law in Hilbert $2$-class fields. The results in this section, as well as our extension question (Question \ref{QUES:CPGT-based-conjecture-prime-splitting-Hilbert-class-fields}), stem from the key ``classical principal genus theorem'' (CPGT) over the rationals.

\begin{theorem}[CPGT over $\QQ$; see {\cite[Proposition 2.12]{Frohlich}}]
\label{THM:classical-PGT-over-rationals}

Let $K/\QQ$ be a quadratic extension, and $\sigma$ a generator of $\Gal(K/\QQ)$. Fix a nonzero fractional ideal $I$ of $K$. Then the ideal class $[I]\in \Cl(K)$ lies in $\Cl(K)^{1-\sigma}$ (which here coincides with $\Cl(K)^2$) if and only if the ideal norm $N_{K/\QQ}(I) = I\sigma(I)\cap \QQ$ takes the form $\alpha\mcal{O}_\QQ$ for some $\alpha\in \QQ^\times$ that is a \emph{local norm} (or \emph{norm residue}) at all (finite and infinite) ramified primes in $K/\QQ$.
\end{theorem}

\begin{remark}
Dominguez, Miller, and Wong \cite{DMW} similarly used Hasse's ``fundamental criterion'' \cite[p. 345]{Hasse} to prove the infinitude of imaginary quadratic fields $K$ with cyclic $2$-class groups $C_{2^n}$ for any $n\ge1$, and Lopez \cite{Lopez} extended their method to $2$-class groups $C_2\oplus C_{2^n}$ for $n\ge1$. Recall that these results give a wealth of examples in Examples \ref{EX:matrix-32-examples} and \ref{EX:matrix-D1-examples}, respectively.
\end{remark}







We now fully work out two specific applications, originally motivated by some SAGE tests based on attempts at modifying Lemmas \ref{LEM:F=Q(p4*p5*)-real-positive-p4*,p5*} (for a concrete example see Example \ref{EX:source-of-failure-and-4-rank-direction-motivation}) and \ref{LEM:F=Q(p3*p4*p5*)-imaginary}.

\begin{theorem}[Cf. Dominguez--Miller--Wong {\cite[proof of Lemma 2.3]{DMW}}]
\label{THM:CPGT-F=Q(p4*p5*)-real-positive-p4*,p5*}

Fix distinct primes $\ell_1,\ell_2\equiv1\pmod{4}$. Let $F = \QQ(\sqrt{\ell_1 \ell_2})$, so $\Cl_2(F) \simeq C_{2^n}$ for some $n\ge1$. Let $p$ be a prime with $\tqchar{\ell_1}{p} = \tqchar{\ell_2}{p} = -1$. Then $(p)$ splits into exactly $2$ primes in the extension $F_{(2)}^1 / \QQ$.
\end{theorem}

\begin{proof}
The rational prime $(p)$ splits into $2$ primes $\mf{p}_1,\mf{p}_2$ in $F/\QQ$, of ideal norm $p\ZZ$. Using the Kronecker symbol conditions on $\ell_1,\ell_2$, it is standard to check that neither $p$ nor $-p$ is a local norm at either of the ramified primes $\ell_1,\ell_2$. By Theorem \ref{THM:classical-PGT-over-rationals} applied to $F/\QQ$, the ideal classes $[\mf{p}_i]\in \Cl(F)$ must lie outside of $\Cl(F)^2$, so their orders are divisible by $2^n$, the size of the largest cyclic $2$-subgroup of $\Cl(F)$. Thus $\mf{p}_1,\mf{p}_2$ are inert in $F_{(2)}^1 / F$ by the decomposition law (Proposition \ref{PROP:decomposition-law}).
\end{proof}

\begin{theorem}[Cf. Lopez {\cite[proof of Proposition 2.3]{Lopez}}]
\label{THM:CPGT-F=Q(p3*p4*p5*)-imaginary}

Fix distinct primes $\ell_1,\ell_2,\ell_3\equiv3\pmod{4}$. Let $F = \QQ(\sqrt{\ell_1^* \ell_2^* \ell_3^*})$ be an imaginary quadratic field with $\Cl_2(F) = \Cl^+_2(F) \simeq C_2\oplus C_{2^n}$ for some $n\ge2$; by genus theory and \Redei--Reichardt, \Wlog{} suppose $R_F = \inmatrix{- & 1 & 1 \\ 0 & - & 1 \\ 0 & 0 & -}$. Take a prime $p\nmid \Delta_F$ and any prime $\mf{p}\mid p$ of $F$. Then $\mf{p}$ splits into exactly $2$ primes in the extension $F_{(2)}^1/F$ if and only if $(\tqchar{\ell_1^*}{p}, \tqchar{\ell_2^*}{p}, \tqchar{\ell_3^*}{p}) \in\{(+1,-1,-1),(-1,+1,-1)\}$.
\end{theorem}

\begin{proof}
By the decomposition law (Proposition \ref{PROP:decomposition-law}), $\mf{p}$ splits into exactly $2$ primes in $F_{(2)}^1/F$ if and only if the ideal class $[\mf{p}]\in \Cl(F)$ has order divisible by $2^n$; if and only if the $\Cl(F)[2]$-coset $[\mf{p}]\cdot \Cl(F)[2]$ is disjoint from $\Cl(F)^2$.

But since $F$ is imaginary, genus theory says the $2$-torsion $\Cl(F)[2] \le \Cl_2(F)$ is generated by the (ideal classes of the) ramified primes $L_1,L_2,L_3$ of $F$, defined by $(\ell_i) = L_i^2$. So by Theorem \ref{THM:classical-PGT-over-rationals}, $\mf{p}$ (of ideal norm $p\ZZ$ if $(p)$ splits in $F/\QQ$, and $p^2\ZZ$ if $(p)$ is inert) splits into exactly $2$ primes in $F_{(2)}^1/F$ if and only if $(p)$ splits in $F/\QQ$ and $\eps p \prod \ell_i^{t_i}$ is never (for any choice $\eps=\pm1$ and $t_i\in\FF_2$) a local norm at all four ramified primes $\infty,\ell_1,\ell_2,\ell_3$. By standard computations (for instance, based on Hensel's lemma or Hilbert symbols), this is the case unless and only unless $(p)$ is inert or $(\tqchar{\ell_1^*}{p}, \tqchar{\ell_2^*}{p}, \tqchar{\ell_3^*}{p}) \in\{(+1,+1,+1),(-1,-1,+1)\}$; in other words, if and only if $(\tqchar{\ell_1^*}{p}, \tqchar{\ell_2^*}{p}, \tqchar{\ell_3^*}{p}) \in\{(+1,-1,-1),(-1,+1,-1)\}$, as desired.
\end{proof}

We now raise a natural extension question of independent interest.

\begin{question}
\label{QUES:CPGT-based-conjecture-prime-splitting-Hilbert-class-fields}

Let $F/\QQ$ be a quadratic field with prime discriminant factorization $\Delta_F = \ell_1^*\cdots\ell_t^*$, with $t\ge1$. Take a rational prime $p\nmid \Delta_F$, and any prime $\mf{p}\mid p$ in $F$. Then is it true that the condition ``the ideal class $[\mf{p}]$ has order divisible by the largest possible $2$-power order in $\Cl(F)$'' depends only on the prime discriminants $\ell_1^*,\ldots,\ell_t^*$ and the Kronecker symbols $\tqchar{\ell_i^*}{p}$ for $1\le i\le t$? What if we work with the narrow class group $\Cl^+(F)$ instead of $\Cl(F)$? If this is false, is there an interesting correct statement along these lines?
\end{question}

\begin{remark}
The methods from Theorems \ref{THM:CPGT-F=Q(p4*p5*)-real-positive-p4*,p5*} and \ref{THM:CPGT-F=Q(p3*p4*p5*)-imaginary} only apply when the $2$-class group $\Cl_2(F)$ is a direct sum of finitely many copies of $C_2$ and $C_{2^n}$, for some $n\ge1$, and the $2$-torsion $\Cl(F)[2]$ is generated by the ideal classes of the ramified primes in $F/\QQ$. In this special case, the order divisibility condition in question actually only depends on the Kronecker symbols $\tqchar{\ell_i^*}{\ell_j}$ and $\tqchar{\ell_i^*}{p}$ for distinct indices $i,j\in\{1,\ldots,t\}$.

However, based on SAGE data (not assuming any conjectures such as GRH), this stronger correspondence does not hold in general; for example, take $F$ with $t=4$ and $(\ell_1^*,\ell_2^*,\ell_3^*,\ell_4^*) = (+5,+29,+109,\ell_4)$ with all $\ell_i\equiv1\pmod{4}$; then the patterns differ for $\ell_4 = 661$ and $\ell_4 = 2609$ although the \Redei matrices $R_F$ do not. (In both cases, SAGE says $\Cl(F) \simeq C_8\oplus C_4\oplus C_2$ and $\Cl^+(F) \simeq C_8 \oplus C_4\oplus C_4$.)
\end{remark}

\begin{markednewpar}\leavevmode\end{markednewpar}\begin{arxiv}
\lstset{language = Python, basicstyle  = \ttfamily, breaklines=true}
\begin{lstlisting}
R.<x> = PolynomialRing(QQ)
p1 = +5; p2 = +29; p3 = +109
#p4 = +661
p4 = +2609
F.<d> = (x^2 - p1*p2*p3*p4).splitting_field(); F
print F.class_group(); print F.narrow_class_group()
Cl = F.class_group()

PrimeCount = 0; OrderCount = [0 for j in range(100)] #keep track of distribution of orders

for i in range(100000):
    p = i
    if is_prime(p) and (kronecker_symbol(p1,p) == +1) and (kronecker_symbol(p2,p) == -1) and (kronecker_symbol(p3,p) == -1) and (kronecker_symbol(p4,p) == +1): #exactly orders 2,4 occur for 661; exactly order 8 occurs for 2609
        PrimeCount += 1
        J = F.ideal(p); #J
        I = J.prime_factors()[0]; #J.prime_factors()
        OrderI = Cl(I).order(); #OrderI
        OrderCount[OrderI] += 1

print PrimeCount
print OrderCount
\end{lstlisting}
\end{arxiv}


\begin{remark}
It may be helpful to consider the criteria from Waterhouse \cite{Waterhouse} (see also Hasse \cite{Hasse} and Lu \cite{Lu}) and Kolster \cite{Kolster} for ideal classes to be $4$th powers, and so on.
\end{remark}

\section{Further directions}
\label{SEC:improvements-of-GS?}

Perhaps one can simply apply Propositions \ref{PROP:main-idea-extension-ramification-quadratic}, \ref{PROP:Schmithals-2-class-field-splitting-idea}, or close variants in cleverer ways, using constructions related to genus fields, relative genus fields (cf. Cornell \cite{Cornell}), decomposition fields (cf. Mouhib \cite{Mouhib}), or narrow Hilbert class fields. Remark \ref{RMK:non-Galois-choices-of-L?} suggests it may also help to choose non-Galois fields $L/\QQ$ in Proposition \ref{PROP:main-idea-extension-ramification-quadratic} or \ref{PROP:Schmithals-2-class-field-splitting-idea}.

It may also be possible to use the group-theoretic classification methods of Benjamin--Lemmermeyer--Snyder \cite{BLS}/Boston--Nover \cite{BN}/Bush \cite{Bush} (see also Boston's survey \cite{BostonSurvey}), where the extra structure may help us push past the ``near miss'' Golod--Shafarevich failures encountered in (for instance) Sections \ref{SEC:key-lemmas} and \ref{SEC:concrete-new-2-towers-results}.

Or perhaps we need better machinery: in view of older results of Koch--Venkov \cite{KV}/Schoof \cite{Schoof}/Maire \cite{Maire}, Koch \cite{KochCentral,KochConcrete,Koch}\begin{arxiv} (see also Lemmermeyer \cite[Section 1.9.4. Galois groups of class field towers]{LemmermeyerCFT})\end{arxiv}, and Gasch\"{u}tz--Newman \cite{GN}---based on the Zassenhaus filtration of group algebras over $\FF_2$\begin{arxiv} (see also Jennings \cite{Jennings}, the book \cite{NSW}, and McLeman \cite{McLeman})\end{arxiv}, and the study of central extensions\begin{arxiv} (see also \Frohlich's book \cite{Frohlich} and Horie \cite{Horie})\end{arxiv}---we ask the following.

\begin{question}
\label{QUES:stronger-4-rank-Golod-Shaf?}

Are there stronger yet usable versions of the Golod--Shafarevich inequality for a number field $K$ cleanly incorporating, for instance, the $4$-rank $d_4 \Cl(K)$? Results in special cases could still be useful.
\end{question}

We now suggest how one might apply such a strengthening to Martinet's question.

\begin{remark}
Recall, in the notation of Proposition \ref{PROP:main-idea-extension-ramification-quadratic}, the failure of our methods in Example \ref{EX:source-of-failure-and-4-rank-direction-motivation} due to $\rank_2\Cl(L(\sqrt{\Delta_K})) = 7 < 8 = 2 + 2\sqrt{8+1}$ (based on SAGE computations assuming GRH); it is not the exact example that matters, but rather that these ``near misses'' occur all the time when one tries to directly use Golod--Shafarevich, which is based only on $2$-rank. So an affirmative answer to Question \ref{QUES:stronger-4-rank-Golod-Shaf?} might allow us to incorporate Remark \ref{RMK:improvement-4-rank-relative-genus-theory} on $4$-rank information from the ambiguous class number formula, or Yue's generalized \Redei matrix $4$-rank criterion \cite{Yue} when $L$ has odd class number.
\end{remark}








\section{Acknowledgements}
This research was conducted at the University of Minnesota Duluth REU and was supported by NSF grant 1358695 and NSA grant H98230-13-1-0273. The author thanks Joe Gallian for suggesting the problem and a careful proofreading of the manuscript; Ian Whitehead for encouragement, a thorough reading of the paper, providing Example \ref{EX:matrix-C-examples}, and insightful mathematical and organizational comments; Elliot Benjamin for clarifying some typos in \cite[Questions 4 and 5]{Ben2015}; Farshid Hajir for sending a copy of \cite{HajirCorrection}; David Moulton for helpful discussions; and Bjorn Poonen for teaching helpful algebraic number theory courses at MIT in the Fall 2014 and Spring 2015 semesters. The author is indebted to the anonymous referee for a thorough reading and many helpful corrections and suggestions, many of which have been incorporated into this version of the paper. The author would also like to thank Yutaka Sueyoshi \cite{Sueyoshi5neg,Sueyoshi3neg} and Elliot Benjamin \cite{Ben2015} for their thorough analysis of \Redei matrices (in some sense identifying the ``hardest cases'' remaining), and Franz Lemmermeyer for his excellent surveys and expositions on a wide variety of class field theory topics and his English translation of \cite{KochCentral}. Finally, the author thanks \href{https://cloud.sagemath.com/}{SageMathCloud} and \href{https://sagecell.sagemath.org/}{SageMathCell} for enabling class group computations\begin{arxiv}, Google Scholar `allintitle: 2-class group' for the serendipitous pointers to \cite{DMW} and \cite{Lopez},\end{arxiv} and \href{https://www.overleaf.com/read/rcfjpsqrxvgt}{Overleaf} for speeding up the writing process.


\end{document}